\documentclass[11pt,a4paper]{amsart}
\usepackage[utf8]{inputenc}
\usepackage{amsmath}
\usepackage{amsthm}
\usepackage{import}
\usepackage{geometry}
\usepackage{mathrsfs} 
\usepackage{amsfonts}
\usepackage{tkz-graph}
\usepackage{amssymb}
\usepackage{palatino}
\usepackage{hyperref}
\usepackage{tikz}
\usepackage{float}
\usetikzlibrary{matrix}
\usepackage{comment}
\usepackage{enumerate,graphicx}
\usepackage{tikz}
\usetikzlibrary{arrows,shapes,positioning}
\usetikzlibrary{decorations.markings}
\usetikzlibrary{calc}

\newtheorem{mydef}{Definition}[section]

\newtheorem{theorem}{Theorem}[section]
\newtheorem{lemma}{Lemma}[section]

\newtheorem{remark}{Remark}[section]

\newtheorem{assumption}[theorem]{Assumption}

\newcommand{\IND}{\mathbf{1}}

\newcommand{\R}{\mathbb{R}}

\newcommand{\RD}{\mathbb{R}^d}
\newcommand{\N}{\mathbb{N}}
\newcommand{\Z}{\mathbb{Z}}
\newcommand{\cA}{\mathcal{A}}

\newcommand{\cC}{\mathcal{C}}
\newcommand{\cF}{\mathcal{F}}
\newcommand{\cE}{\mathcal{E}}

\newcommand{\cP}{\mathcal{P}}

\newcommand{\cX}{\mathcal{X}}

\newcommand{\cT}{\mathcal{T}}

\newcommand{\tb}{\tilde{b}}
\newcommand{\scrA}{\mathscr{A}}
\newcommand{\scrU}{\mathscr{U}}
\newcommand{\frkA}{\mathfrak{A}}
\newcommand{\bes}{\begin{equation*}}
\newcommand{\ees}{\end{equation*}}
\newcommand{\beas}{\begin{eqnarray*}}
\newcommand{\eeas}{\end{eqnarray*}}
\newcommand{\bea}{\begin{eqnarray}}
\newcommand{\eea}{\end{eqnarray}}
\newcommand{\be}{\begin{equation}}
\newcommand{\ee}{\end{equation}}

\newcommand{\bbP}{\mathbb{P}}
\newcommand{\bbE}{\mathbb{E}}
\newcommand{\bbQ}{\mathbb{Q}}
\newcommand{\bbS}{\mathbb{S}}
\newcommand{\bbW}{\mathbb{W}}

\newcommand{\bh}{\mathbf{h}}

\newcommand{\bff}{\mathbf{f}}
\newcommand{\bp}{\mathbf{p}}

\newcommand{\bw}{\mathbf{w}}

\newcommand{\Int}{[0,1]}
\newcommand{\ind}{\mathbf{1}}

\newcommand{\OUBR}{{}^\alpha\bbP^{xy}}

\newcommand{\BRG}{\bbP^{xy}}

\newcommand{\zz}{z \to z'}
\newcommand{\jj}{\bar{j}}

\newcommand{\cc}{\mathbf{c}}
\newcommand{\bfe}{\mathbf{e}}

\newcommand{\SRW}{\bbS}

\newcommand{\qform}{\Sigma^{\alpha}_{I}}
\newcommand{\qmform}{\Sigma^{\alpha}_{\Pi^{<t}_m}}
\newcommand{\Hess}{\mathbf{Hess}}
\newcommand{\WRG}{\bbW^{xy}}
\newcommand{\ces}{\cE^*}
\author{Giovanni Conforti}
\date{February,22 2016}

\begin{document}

\title{Fluctuations of bridges, reciprocal characteristics and concentration of measure}
\newcommand{\Addresses}{{
  \bigskip
  \footnotesize

  Giovanni~Conforti, \textsc{Department of Mathematics, Universit\"at Leipzig,
    Germany}\par\nopagebreak
  \textit{E-mail address},  \texttt{giovanniconfort@gmail.com}

}}
\maketitle
\begin{abstract}
Conditions on the generator of a Markov process to control the fluctuations of its bridges are found. In particular, continuous time random walks on graphs and gradient diffusions are considered. Under these conditions, a concentration of measure inequality for the marginals of the bridge of a gradient diffusion and refined large deviation expansions for the tails of a random walk on a graph are derived. In contrast with the existing literature about bridges, all the estimates we obtain hold for non asymptotic time scales. New concentration of measure inequalities for pinned Poisson random vectors are also established. The quantities expressing our conditions are the so called \textit{reciprocal characteristics} associated with the Markov generator.
\end{abstract}
\tableofcontents

\section{Introduction}
In this paper we study quantitatively bridges of Markov processes over the time-interval $\Int$. As a guideline of our investigations, independently from the details of the model, we have in mind the sketch of the motion of a bridge as divided into two symmetric phases: at first one observes an expansion phase, in which the bridge, starting from its deterministic initial position, increases its randomness. After time $1/2$, a second contraction phase takes place, in which the damping effect of the pinning at the terminal time is so strong that randomness decreases, and eventually dies out. Moreover, one also expects that the two phases enjoy some symmetry with respect to time reversal. To summarize, one can say that the motion of  bridge resembles that of an accordion.
The aim of this paper is to obtain a quantitative explanation of this picture. This means that we consider a Markov process and try to understand
how its semimartingale characteristics should look like in order to observe bridges where the influence of pinning over randomness is stronger than that of a reference model, for which  computations can be carried out in explicit form.
This problem, although quite natural, seems to have received very little attention so far. As we shall see, some precise answers can be given. It is interesting to note that the quantities expressing our conditions are not related to those used to measure the speed of convergence to equilibrium, as one might expect at first glance ( see Remark \ref{remgamma2} for a comparison with the $\Gamma_2$ condition of Bakry and \'Emery \cite{BAKEM} ).\\
There are many possible quantities that could be used to estimate the balance of power between pinning and randomness and make precise mathematical statements. Some of them are discussed in the sequel, depending on the model: in this paper, Brownian diffusions with gradient drift and continuous time random walks on a graph are considered.\\
Our results take the form of comparison theorems, which yield quantitative information on the bridge at non asymptotic  time  scales. In Theorem \ref{accordeon} we find conditions on the potential of a gradient diffusion for its bridges to have marginals  with better concentration properties than those of an Ornsetin Uhlenbeck bridge. This is one of the main novelties with respect to the existing literature about bridges where, to the best of our knowledge, only Large Deviations-type estimates have been proved, and mostly in the short time regime, see among others \cite{bailleul2013large}, \cite{bailleul2015small}, \cite{baldi2002asymptotics}, \cite{baldi2014large},  \cite{dawson1990schrodinger}, \cite{privault2015large}, \cite{wittich2005explicit} and  \cite{yang2014large}.  The proof of this result is done by first showing an ad hoc Girsanov formula for bridges, which differs from the usual one. We then employ some tools developed in  \cite{BrLieb76} to transfer log concavity of the density from the path space to the marginals, and use the well known properties of log concave distributions. \\
Theorems \ref{squarelatticebound} and \ref{treebound} concern continuous time random walks on graphs with constant speed: we find conditions on the jump rates under which the marginals of the bridge have lighter tails than those of the simple random walk. For their proof we rely on some elementary, though non trivial combinatorial constructions that allow to control the growth of the \textit{reciprocal characteristics} associated with the cycles of the graph, as the length of the cycles increases.
The study of bridges of continuous random walks brings naturally to consider pinned Poisson random vectors: we derive concentration of measure inequality for these distributions, using a Modified Log Sobolev Inequality and an interpolation argument.  
\subsection*{Reciprocal characteristics} An interesting aspect is that the conditions which we impose to derive the estimates are expressed in terms of the so called  \textit{reciprocal characteristics}.
 The reciprocal characteristics of a Markov processare a set of invariants which fully determine the family of bridges associated with it:
such a concept has been introduced by Krener in \cite{Kre88}, who was interested in developing a theory of stochastic differential equations of second order, motivated by problems in Stochastic Mechanics.
Several authors then contributed to the development of the theory of reciprocal processes and second order differential equations. Important contributions are those of  Clark \cite{Cl91}, Thieullen \cite{Th93}, L\'evy and Krener \cite{LK96}, and Krener \cite{Kre97}. Roelly and Thieullien in \cite{RT02}, \cite{RT05} introduced a new approach based on integration by parts formulae. This approach was used to study reciprocal classes of continuous time jump processes in \cite{CLMR}, \cite{CDPR}, \cite{CR}.
The precise definitions are given at Definition \ref{def:recchardiff} and \ref{def:recchgr} below. However, let us give some intuition on why they are  an interesting object to consider to bring some answers to the aforementioned problems. For simplicity, we assume that $\bbP^x$ is a diffusion with drift $b$ and unitary dispersion coefficients.
It is well known that the bridge $\bbP^{xy}$ is another Brownian diffusion with unitary diffusion matrix and whose drift field $\tb$ admits the following representation:
\bes
\tb (t,z) = b(t,z) + \nabla \log h(t,z), 
 \ees
 where $h(t,z)$ solves the Kolmogorov backward PDE:
 \bes
 \partial_t h(t,z) + b \cdot \nabla h(t,z) + \frac{1}{2} \Delta h(t,z) = 0 , \quad \lim_{t \uparrow 1}h(t,z)= \IND_{z=y} 
 \ees
This is the classical way of looking at bridges as $h$-transforms, which goes back to Doob \cite{Doob1957}. However, it might not be the most convenient one to perform explicit computations. The first reason is that $h$ is not given in explicit form. Moreover, this representation does not account for the time symmetric nature of bridges. Actually, the problem of restoring this time symmetry was one of the motivations for several definitions of conditional velocity and acceleration for diffusions in the context of stochastic mechanics, see e.g. \cite{Nel67}, \cite{CruZa91}, \cite{Th93}.
The theory of reciprocal processes proposes a different approach  to bridges: there one looks for a family of (non-linear) differential operators $\frkA$ with the property that the system of equations
\bes
\mathscr{A} \tb = \mathscr{A} b, \quad \scrA \in \frkA
\ees
together with some boundary conditions characterizes the drift $\tb$ of $\bbP^{xy}$. For diffusions, they were computed for the first time by Krener in \cite{Kre88}, and subsequently used by Clark \cite{Cl91} to characterize reciprocal processes.

For instance, in the case of 1-dimensional Brownian diffusions we have $\frkA = \left\{ \scrA \right\}$ with 
\bes
\scrA b = \frac{1}{2} \partial_{xx} b + b \partial_x b +  \partial_t b
\ees
The advantage of this approach is to show that the drift of the bridge $\tb$ depends on $b$ only through the subfields  $\scrA b$, for $\scrA \in \frkA$, and not on anything else. In other words: two different processes with the same reciprocal characteristics share have identical bridges (for results of this type, see \cite{Blee},   \cite{Fitz}, \cite{CLMR}, \cite{RT02}, \cite{RT05}, \cite{CDPR}, \cite{CR}, \cite{CL15}). Therefore, one sees that any optimal condition to control the fluctuations of $\bbP^{xy}$ shall be formulated in terms of the characteristics since other conditions will necessarily involve some features of $b$ which play no role in the construction of $\bbP^{xy}$. This simple observation already rules out some naive approaches to the problems studied in this paper. Indeed one might observe that when $\bbP^x$ is time homogeneous we have:
\bes
\BRG(X_t \in dz ) \propto \bbP^x(X_t \in z+ dz) \, \bbP^{z}(X_{1-t} \in y+ dy )
\ees
 and then an optimal criterion to control the fluctuations of the marginals of $\bbP$ suffices. But since any known condition to bound them is not expressed in terms of the reciprocal characteristics, this strategy has to be discarded. 
Reciprocal characteristics enjoy a probabilistic interpretation: they appear as the coefficient of the leading terms in the short time expansion of either the conditional probability of some events ( see \cite{CL15} for the discrete case) or the conditional mean acceleration (see \cite{Kre97} in the diffusion setting).
 Indeed, one can view the results of this article as the global version of the "local" estimates which appear in the  works above. A first result in this direction has been obtained in \cite{conforti2016counting}, where a comparison principle for bridges of counting processes is proven.
Reciprocal characteristics have been divided into two families, \textit{harmonic characterisitcs} and \textit{rotational (closed walk)} characteristics.
We discuss the role of harmonic characteristics in the diffusion setting and the role of rotational characteristics for continuous time random walks on graphs.
\subsubsection*{Organization of the paper}
In Section 2 and 3 we present our main results for diffusions and random walks. They are main results which are Theorem \ref{accordeon}, Theorem \ref{t70}, Theorem \ref{squarelatticebound} and Theorem \ref{treepatch}. Section $4$ is devoted to proofs. We collect in the Appendix some results on which we rely for the proofs.
\subsubsection*{General notation}
We consider Markov processes over $[0,1]$ whose state space $\cX$ is either $\R^d$ or the set of vertices of a countable directed graph.
We always denote by $\Omega$ the cadl\'ag space over $\cX$, by $(X_t)_{0\leq t \leq 1}$ the canonical process, and by $ \cP(\Omega)$ the space of probability measures over $\Omega$ . On $\Omega$ a Markov probability measure $\bbP$ is given, and we study its bridges. In our setting, bridges will always be well defined for \textit{every} $x,y \in \cX^2$ and not only in the almost sure sense. We will make clear case by case why this is possible. As usual $\bbP^x$ is $\bbP( \cdot | X_0 =x)$, $\bbP^{xy}$ is the $xy$ bridge, $\bbP^{xy} := \bbP(\cdot | X_0 =x, X_1 =y)$. For $I\subseteq \Int$, we call $X_{I}$ the collection $(X_t)_{t \in I}$ and the image measure of $X_{I}$ is denoted $\bbP_{I}$. Similarly, we define $\bbP^x_{I}$, and $\bbP^{xy}_{I}$. For a general  $\bbQ \in \cP(\Omega)$, expectation under $\bbQ$ is denoted $\bbE_{\bbQ}$. We use the notation $\propto$ when two functions differ only by a multiplicative constant.
\section{Bridges of gradient diffusions: concentration of measure for the marginals}
 
 \subsubsection*{Preliminaries}
We consider gradient-type diffusions. The potential $U$ is possibly time dependent and satisfies one among hypothesis (2.2.5) and (2.2.6) of Theorem 2.2.19 in \cite{royer2007initiation}, which ensure existence of solutions for
\be\label{eq:grdiff}
d X_t = - \nabla U (t,X_t)dt + dB_t, \quad X_0 =x.
\ee
Bridges of Brownian diffusions are well defined for any $x,y \in \R^d$. This fact is ensured by \cite[Th.1]{chaumont2011markovian} and the fact that $\bbP$ admits a smooth transition density. 
A special notation is used for Ornstein-Uhlenbeck processes. We use  $^{\alpha}\bbP^{x}$ for the law of :
\be\label{eq:OUprc}
d X_t = - \alpha \, X_t  dt  + d B_t, \quad X_0 = x 
\ee
where $\alpha>0$ is a positive constant. $\OUBR$ is then the $xy$ bridge of $^{\alpha}\bbP^{x}$
Let us give some standard notation. For $v \in \RD$, $v^T$ is the transposed vector. If $w$ is another vector in $\RD$, we denote the inner product of $v$ and $w$ by $v \cdot w$. Similarly, if $H$ is a matrix and $v$ a vector, the product is denoted $H \cdot v$. The Hessian matrix of a function $U$ is denoted $\Hess(U)$, and by $\Hess(U) \geq \alpha \, \mathbf{id} $ we mean, as usual, that
\bes
\inf_{v: v \cdot v =1 } v^T \cdot \Hess(U)(z) \cdot v \geq \alpha
\ees
The norm of $v\in \RD$ is $\| v\|$.
 Let us now give definition of reciprocal characteristics for gradient diffusions. It goes back to Krener \cite{Kre88}.
 \begin{mydef}\label{def:recchardiff}
Let $U: \Int \times \RD \rightarrow \R$ be a smooth potential. We define $\scrU: [0,1] \times \RD \rightarrow \R$ as:
\be\label{e16}
\scrU (t,z) := \big[  \frac{1}{2} \| \nabla U\|^2 - \partial_t U - \frac{1}{2}\Delta U \big] (t,z) 
\ee
 The \underline{harmonic characteristic} associated with $U$ is the vector field $\nabla \scrU$.
  \end{mydef}
\subsubsection*{Measuring the fluctuations}
Consider the bridge-marginal $\bbP^{xy}_t$. We denote its density w.r.t. to the Lebesgue measure by $p^{xy}_t(z)$. As an indicator for the "randomness" of $\bbP^{xy}_t$ we use $\gamma(t)$, defined by:
\be\label{e90}
\gamma(t) = \sup \{ \beta : -\Hess(\log p^{xy}_t)(z) \geq \beta \mathbf{id} \}
\ee
It is well known that lower bounds on $\gamma(t)$ translate into concentration properties for $\bbP^{xy}_t$, see Theorem 2.7 of \cite{Led01}. The better the bound, the stronger the concentration. 
In the Ornstein Uhlenbeck case, $\gamma(t):= \gamma_{\alpha}(t)$ can be explicitly computed. The actual computation will be carried out in the proof of Theorem \ref{accordeon}. We have:
\be\label{e40}
\gamma_{\alpha}(t) =\frac{2\alpha(1-\exp(-2\alpha))}{(1-\exp(-2 \alpha t)) (1-\exp(-2 \alpha(1-t) ) )} 
\ee
Note that $\gamma_{\alpha} $ obeys few stylized facts:

\begin{enumerate}[(i)]
\item 
It is symmetric around $1/2$: this reflects the time symmetry of the bridge.
\item It converges to $+ \infty$ as $t$ converges to either $0$ or $1$. This is due to the pinning.
\item  $\gamma_{\alpha}$ is convex in $t$. This also agrees with the description of the dynamics of a bridge we sketched in the introduction. Convexity reflects the fact that, as time passes, the balance of power between pinning and randomness goes in favor pinning , whose impact on the dynamics grows stronger and stronger, whereas 
the push towards randomness stays constant, since the diffusion coefficient does not depend on time.
\item It is increasing in $\alpha$.
\end{enumerate}
 Theorem \ref{accordeon} is a comparison theorem for $\gamma(t)$. We show that if the Hessian of $\scrU$ (see \eqref{e16}) enjoys some convexity lower bound, say $\frac{1}{2}\alpha^2$, then $\gamma(t)$ lies above $\gamma_{\alpha}(t)$: this means that $\bbP^{xy}$ is more concentrated than $^{\alpha}\bbP^{xy}$. 
\begin{theorem}\label{accordeon}
Let $\bbP^x$ be the law of \eqref{eq:grdiff} and $\scrU$ be defined at \eqref{e16}.
If, uniformly in $r \in [0,1], z \in \RD$:
\be\label{e17}
\Hess( \scrU )(r,z) \geq \frac{\alpha^2 }{2} \, \mathbf{id}
\ee
then the following estimate holds for any $t \in [0,1]$, and any $1$-Lipschitz function $f$:
\bes
 \bbP^{xy}\Big(f(X_t) \geq\bbE_{\bbP^{xy}}(f(X_t)) + R \Big) \leq \exp \Big(-\frac{1}{2}{ }\gamma_{\alpha}(t) R^2\Big) 
\ees
where $\gamma_{\alpha}(t)$ is defined at \eqref{e40}.
\end{theorem}

The proof of Theorem \ref{accordeon} uses three main tools:  the first one is an integration by parts formula for bridges of Brownian diffusions due to Roelly and Thieullen, see \cite{RT02} and \cite{RT05}. Such formula has the advantage of elucidating the role of reciprocal characteristics, and we reported it in the appendix. 
The second one is a statement about the preservation of strong log concavity due to Brascamp and Lieb \cite{BrLieb76}. This theorem is a quantitative version of the well known fact that marginals of log concave distributions are log concave. We refer to Remark \ref{noway} for more comparison between Theorem \ref{accordeon} with some of the results of \cite{BrLieb76}.
Finally we will profit from the well known concentration of measure properties of log concave distributions, for which we refer to \cite[Chapter 2]{Led01}.
\begin{remark} 
The condition  \eqref{e17} does not depend on the endpoints $(x,y)$ of the bridge
\end{remark}

\begin{remark}
The estimates obtained here are sharp, as the Ornstein Uhlenbeck case demonstrates: a simple computations shows that $\frac{\alpha^2}{2}  \mathbf{id} $ is indeed the Hessian of $\scrU$ when $\bbP^x= ^{\alpha}\bbP^x$. 
\end{remark}
\begin{remark}\label{remgamma2}
The $\Gamma_2$ condition of Bakry and \'Emery in this case reads  as
\bes
\Hess(U) \geq \alpha \, \mathbf{id}
\ees
which is clearly very different from \eqref{e17}. In particular, \eqref{e17} involves derivatives of order  up to four. However, a simple manipulation of Girsanov's theorem formally relates the two conditions. 
Consider the density  $M$ of $\bbP^x$ with respect to the Brownian motion started at $x$. We have by Girsanov's formula (for simplicity, we assume $U$ not to depend on time):
\bes
M= \exp\left(- \int_{0}^1 \nabla U(X_t) \cdot dX_t - \frac{1}{2} \int_{0}^1 \| \nabla U \|^2(X_t) dt \right)
\ees
A standard application of It\^o formula allows to rewrite $M$ as:
\bes
\exp\left(- \underbrace{U(X_1)}_{\Gamma_2} + U(x) - \frac{1}{2} \int_{0}^1 \underbrace{\| \nabla U \|^2(X_t) - \Delta U(X_t) }_{ = 2 \scrU}  dt \right)
\ees
Imposing convexity on the first term, one obtains $\Gamma_2$, whereas imposing convexity on the integrand, yields \eqref{e17}. In this sense, the two condition are complementary: what is "seen" from one, is not seen from the other, and viceversa.
\end{remark}
\begin{remark}
Many authors have investigated Logarithmic Sobolev inequalities for the Brownian bridge as a law on path space, or, more generally for the Brownian motion on loop spaces, see e.g. \cite{fang1999integration}. Therefore, starting from those inequalities one should be able to obtain concentration of measure results for the Brownian bridge. Our approach is not based on such inequalities because, to the best of our understanding, they are limited to the bridge of the Brownian motion and we consider bridges of gradient-type SDEs. Moreover, we do not know how precise the concentration bounds derived from these SDEs would be concerning the marginals and there does not seem to be a criterion to construct measures which have concentration properties at least as good as the Brownian bridge measure. This is exactly what we do in this paper. On the other hand, these inequalities are available for curved spaces, a case which we do not touch.
\end{remark}
\section{Continuous time random walks}
In this section we prove various estimates for the bridges of continuous time random walks with constant speed. These estimate are obtained by imposing conditions on the \textit{closed walk characteristics} associated with the random walk. It is shown in \cite[Th. 2.4]{CL15} that the closed walk characteristics of a constant speed random walk fully determine its bridges.
\subsubsection*{Preliminaries}
Let $\cX$ be a countable set and $ \mathcal{A} \subset\cX^2$.
 The \emph{directed graph}  associated with $\mathcal{A}$ is defined by means of the relation $\to$ . For all $z,z'\in\cX^2$ we have $\zz$ if and only if $(z,z')\in \mathcal{A}.$   We denote $(\cX^2,\to)$ this directed graph, say that any $(z,z')\in \mathcal{A}$ is an arc and  write $(\zz)\in \mathcal{A}$ instead of $(z,z')\in \mathcal{A}.$  For any $n\ge1$ and $x_0,\dots,x_n\in\cX$ such that $x_0\to x_1$, $x_1\to x_2,\ \cdots,\ x _{ n-1}\to x_n$, the ordered sequence
 $(x_0,x_1,\dots,x_n)$ is called a \emph{walk}. 
We adopt the notation $	
\bw=(x_0\to x_1 \to \cdots\rightarrow x_n). 
$	When $x_n=x_0$,  the  walk $\cc=(x_0\to x_1 \to \cdots\rightarrow x_n=x_0)$ is said to be \textit{closed}. The  length $n$ of $\bw$  is denoted by $\ell(\bw).$
We introduce a continuous time random walk $\bbP^x$  with intensity of jumps $j:\cA \rightarrow \R_{+}$. $j(z \to z')$ is the rate at which the walk jumps for $z$ to $z'$. To ensure existence of the process, we make some standard assumptions on $j$, and $(\cX,\to)$, which are detailed at Assumption \ref{as-03} and Assumption\ref{as-01}. 
These assumptions also ensure that the bridge is defined between any pair of vertices $x,y \in \cX$.
In this paper, we consider constant speed random walks (CSRW). This means that the function $z \mapsto \jj(z) = \sum_{z': \zz} j(\zz)$ is a constant. 
 Let us define the closed walk characteristics associated with $j$. We refer to \cite{CDPR},  \cite{CL15}, \cite{CR} for an extensive discussion.
 \begin{mydef}\label{def:recchgr}
Let $(\cX,\to)$ be a graph satisfying Assumption \ref{as-03} and $j$ be a jump intensity satisfying Assumption \ref{as-01}. For any $t \in (0,1)$ and any closed walk $\cc=(x_0\to \cdots\to x_{n}=x_0)$   we define the corresponding \underline{closed walk characteristic} as:
\be\label{eq30}
\Phi_j(\cc) := \prod_{i=0}^{n-1} j(x_i \to x_{i+1}) .
\ee
\end{mydef}
\subsection{Concentration of measure for pinned Poisson random vectors}\label{sub1}
\subsubsection*{A simple question}
We fix $k \in \N$ and consider the graph $(\cX, \to)$ where $\cX = \Z$, and $z \to z' $ if and only if $z' = z -1$ or $z' = z+k$. We consider a random walk $\bbP$ with time and space-homogeneous rates:
\bes j(z \to z+k) \equiv j_k , \quad j(z \to z-1) \equiv j_{-1} \quad \forall z \in \Z \ees
The simple\footnote{see Definition \ref{defs-01} for the meaning of simple walk}  closed walks of $(\cX,\to)$ are  of the form 
$$ \cc = (x \to x-1 \to x-2 \to .. \to x-k \to x)$$
for some $x \in \Z$ and, because of the homogeneity of the rates, we  have
$$ \forall  \cc \, \text{simple closed walk}, \quad \Phi_j(\cc) \equiv j^k_{-1}j_k:=\Phi $$
We introduce random variables $ N^{k}$ and $N^{-1}$ which count the number of jumps along arcs of the form $(x \to x+k)$ and $(x \to x-1)$ respectively. Obviously, under $\bbP^0$ the vector $(N^{k},N^{-1}) $ is a two dimensional vector with independent components following a Poisson law of parameter $j_{k}$ and $j_{-1}$ respectively. 
Let us consider the $00$ bridge of $\bbP$, $\bbP^{00}$. The distribution of $N^{k}$ is that of  the first coordinate of a Poisson random vector conditioned to belong to an affine subspace, precisely $\{(n^k, n^{-1}) \in \N^2 : k \, n^k - n^{-1} = 0 \}$. We call this distribution $\rho_{\Phi}$.
\be \label{eq21}
\rho_{\Phi} ( \cdot ) = \bbP^{00}(N^k  \in \cdot ) = \bbP^{0} \left( N^k \in \cdot \Big | k N^{k} - N^{-1}=0\right) 
\ee
We aim at establishing a concentration of measure inequality for $\rho_{\Phi}$. This is very natural in the study of bridges: one wants to know how many jumps of a certain type the bridge performs. The role of pinning against randomness should be visible in the concentration properties of this distribution.
This task is not trivial because $\rho_{\Phi}$ is no longer a Poissonian distribution. This is in contrast with the Gaussian case, where pinning a Gaussian vector to an affine subspace gives back a Gaussian vector. 
To gain some insight on what rates to expect let us recall Chen's characterization of the Poisson distribution (see \cite{Chen}) of parameter $\lambda$, which we call  $\mu_{\lambda}$:
\be\label{eq23}
 \forall f>0 \quad \quad  \lambda \bbE_{\mu_{\lambda}} \big( f(n+1)\big) = \bbE_{\mu_{\lambda}} \big( f(n) n \big)  
\ee
Using \cite[Prop. 3.8]{CDPR},  one finds an analogous characterization for $\rho_{\Phi}$ as the only solution of
\be\label{eq22}
\forall f>0 \quad \Phi \, \bbE_{\rho_{\Phi}} \left( f(n + 1 ) \right) = \bbE_{\rho_{\Phi}} \left( f(n)  n \, \prod_{i=0}^{k-1} (k n - i) \right)
\ee
The density on the right hand side of \eqref{eq22}  is a polynomial of degree $k+1$. By choosing $f(n)=\mathbf{1}_{n=z}$ in both \eqref{eq23} and \eqref{eq22}, we obtain:
\be\label{e120} \forall z \in \N, \quad   \frac{\mu_{\lambda}(z-1)}{\mu_{\lambda}(z) } = \frac{1}{\lambda} z , \quad   \frac{\rho_{\Phi}(z-1)}{\rho_{\Phi}(z)} = \frac{z\prod_{i=0}^{k-1} (k z - i) }{\Phi }\sim \frac{z^{k+1}}{\Phi}  \ee
from which we deduce that $\rho_{\Phi}$ has much lighter tails than $\mu_{\lambda}$. The corresponding concentration inequalities should reflect this fact.
We derived the following result:
\begin{theorem}\label{t70}
Let $ \rho_{\Phi} $ be defined by \eqref{eq21}. Consider a $1$-Lipschitz function $f$. Then,  for all $R>0$:
\be\label{eq24}
\rho_{\Phi}\big(f \geq \bbE_{\rho_{\Phi}}(f)+R \big) \leq  \exp\big(-(k+1)R \log R +[ \log(\Phi) + c ] R + o(R) \big)
\ee
The constant $c$ is a structural constant which depends only on $k$. 
\end{theorem}
In \eqref{eq24}, and in the rest of the paper, by $o(R)$ we mean a function $g$ such that $\lim_{R \rightarrow + \infty } g(R)/R = 0$. The $o(R)$ term in \eqref{eq24} can be made explicit: it depends on $\Phi$ and $k$, but not on $f$. By following careful the proof of this theorem, it is possible to see that the bound \eqref{eq24} is interesting (i.e. the right hand side is $<1$) when $R \geq \Phi + \frac{1}{k+1} \Phi^{1/(k+1)}$. The bound is very accurate for $R$ large, see Remark \ref{Herbstrem}.

\begin{remark}\begin{enumerate}[(i)]
\item The size of the large jump drives the leading order in the concentration rate, while the reciprocal characteristic is responsible for the exponential correction term.
\item The larger $k$, the more concentrated is the random variable. This is because to compensate a large jump a bridge has to make many small jumps, and this reduces the probability of large jumps.
\item The smaller $\Phi$, the better the concentration. This fits with the short time interpretation of $\Phi$  given in \cite[Th.2.7]{CL15}
\end{enumerate}
\end{remark}
\begin{remark}[Sharpness]
It can be seen, using Stirling's approximation and \eqref{e120}  that the leading order term  $ - (k+1)R \log R$  is optimal and  the linear dependence on $\log(\Phi)$ at the exponential correction term is correct.
\end{remark}
The proof of this theorem is based on the construction of a measure $\pi_{\Phi}$ which interpolates $\rho_{\Phi}$ and for which the modified log Sobolev (MLSI) inequality gives sharp concentration results. Several MLSI have been proposed  for the Poisson distribution. We use the one which is considered in \cite{DP02} and \cite[Cor 2.2]{wu2000new}. The reason for this choice is that there are robust criteria (see \cite{CapPos07})  under which such inequality holds. The interpolation argument is crucial to achieve the rate $-(k+1)R \log R$. Indeed, the MLSI cannot yield any better than $-R \log R$ .
While doing the proof, we repeat the classical Herbst's argument for the MLSI, improving on some results of \cite{BOB98} (which were obtained by using a different MLSI).
\subsection{Bridges of CSRW on the square lattice: refined large deviations for the marginals.}
 Let $v_1= (1,0),v_2=(0,1)$. The square lattice is defined by $\cX = \Z^2$ and by saying that the neighbors of $x$ are $x \pm v_1$ and $x \pm v_2$.
We associate to any vertex $x \in \Z^2$ the clockwise oriented face $\bff_{x}$ and two closed walks of length two, $\bfe_{x,1},\bfe_{x,2}$ as follows:
\beas
\bff_x &=& (x \to x+v_2 \to x + v_1 + v_2 \to  x+v_1 \to x) \\
\bfe_{x,1} &=& (x \to x+v_1 \to x), \quad \bfe_{x,2} = ( x \to x + v_2 \to  x )
\eeas 
The set of closed walks of length two is denoted $\cE$:
\be\label{e91}
\cE = \{ (x\to y \to x) : (x \to y ) \in \cA \} = \left\{  \bfe_{x,i}, x \in \cX, i \in \{ 1,2\} \right\}
\ee
The set of clockwise oriented faces is $\cF$:
\bes
\cF:= \left\{ \bff_{x} : \, x \in \Z^2 \right\}.
\ees

In this subsection we prove an analogous statement to Theorem \ref{accordeon} for CSRWs on the square lattice.
A serious difficulty here is represented by the fact that there is not such a well developed theory to prove concentration of measure inequalities with Poissonian rates. In particular, all the tools we use in the proof of Theorem \ref{accordeon} do not have a "Poissonian" counterpart. To the best of our knowledge, the only result concerning Poisson-type deviation bounds for the marginals of a continuous time Markov chain is due to Joulin \cite{Joulin2007Poisson}. In Theorem 3.1 the author provides  abstract curvature conditions under which such bounds hold. However, explicit construction of Markov generators fulfilling these conditions is limited to 1-dimensional birth and death process, see Section 4.
Therefore, instead of using $\gamma(t)$ (see \eqref{e90}) we shall use a simpler way to measure the fluctuations of the bridge, adopting a Large Deviations viewpoint.
We will look at asymptotic tail expansions, and relate the coefficients in the expansion with reciprocal characteristics. This is a much rougher measurement, but still gives interesting results. 
 We consider the $00$ bridge $\bbP^{00}$ of the simple random walk which jumps along any arc with intensity constantly equal to $\lambda$. Using some classical expansions (see Lemma \ref{lastlemma}) one finds that:
\be\label{e121}
\log\Big( \bbP^{00} \big(d(X_t, \bbE_{\bbP^{00}}(X_t)) \geq R  \big)\Big) =  -2 R\log(R) + \big[ \log (4\lambda^2 t(1-t) \,) +2 \big]R + o(R)
\ee
 Theorem \ref{squarelatticebound} provides a condition on the reciprocal characteristics for the \eqref{e121} to hold when replacing $=$ with $\leq$. The conditions are expressed as conditions on the closed walks characteristics associated to the walks in $\cE \cup \cF$.  
\begin{theorem}\label{squarelatticebound}
Let $j:\cA \rightarrow \R_{+}$ be the intensity of a CSRW $\bbP$ on the square lattice. Assume that for some $\lambda>0$:
\be\label{eq8}
\forall x \in \Z^2, i \in \{1,2 \} \quad   \Phi_j(\bfe_{x,i})  \leq \lambda^2
\ee
and
\be\label{eq7}
\forall x \in \Z^2 \quad    \Phi_j(\bfe_{x,2})  \Phi_j(\bfe_{x,1}) \leq   \Phi_j(\bff_x) \leq  \Phi_j(\bfe_{x+v_1,2} )   \Phi_j(\bfe_{x+v_2,1}) 
\ee
then for any $x \in \Z^2$:
\bes
\log \bbP^{xx} \left ( d(X_t,\bbE_{\bbP^{xx}}(X_t)) \geq R \right) \leq -2\,R\log(R) + \big[ \log (4\lambda^2 t(1-t) \,) +2 \big] R + o(R)
\ees
\end{theorem}
\tikzset{middlearrow/.style={
        decoration={markings,
            mark= at position 0.5 with {\arrow{#1}} ,
        },
        postaction={decorate},
        thick,
        red
    }
}
\tikzset{earlyarrow/.style={
        decoration={markings,
            mark= at position 0.3 with {\arrow{#1}} ,
        },
        postaction={decorate},
        thick,
        blue
    }
}
\tikzset{latearrow/.style={
        decoration={markings,
            mark= at position 0.7 with {\arrow{#1}} ,
        },
        postaction={decorate},
        thick,
        blue
    }
}
\tikzset{glatearrow/.style={
        decoration={markings,
            mark= at position 0.7 with {\arrow{#1}} ,
        },
        postaction={decorate},
        thick,
       yellow
    }
}
\tikzset{gearlyarrow/.style={
        decoration={markings,
            mark= at position 0.3 with {\arrow{#1}} ,
        },
        postaction={decorate},
        thick,
     yellow
    }
}
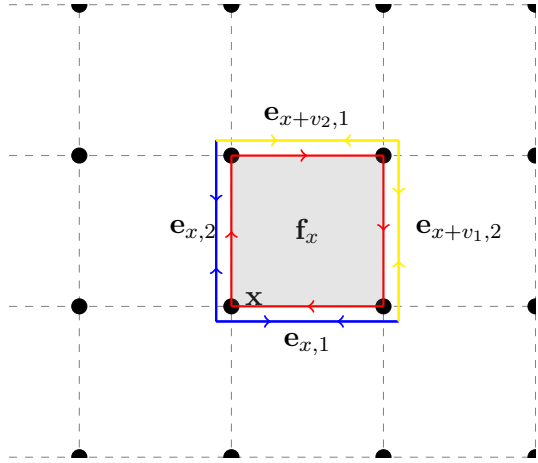
\begin{figure}[h!]
  \centering
  \begin{tikzpicture}
    \coordinate (Origin)   at (0,0);

    \clip (-1,-0) rectangle (6cm,6cm); 
    \pgftransformcm{1}{0}{0}{1}{\pgfpoint{0cm}{0cm}}
    \coordinate (Bone) at (0,2);
    \coordinate (Btwo) at (2,-2);
    \draw[style=help lines,dashed] (-14,-14) grid[step=2cm] (14,14);
    \foreach \x in {-7,-6,...,7}{
      \foreach \y in {-7,-6,...,7}{
        \node[draw,circle,inner sep=2pt,fill] at (2*\x,2*\y) {};
      }
    }
        \node[] at (2.3,2.1) {$\mathbf{x}$};
         \node[] at (3,3) {$\mathbf{f}_{x}$};
         \node[] at (3,1.5) {$\mathbf{e}_{x,1}$};
         \node[] at (1.5,3) {$\mathbf{e}_{x,2}$};
			 \node[] at (5.0,3) {$\mathbf{e}_{x+v_1,2}$};
         \node[] at (3,4.5) {$\mathbf{e}_{x+v_2,1}$};

    \filldraw[fill=gray, fill opacity=0.2] ($(2,2)$)
        rectangle ($(4,4)$);
        \draw[middlearrow={>}] (2,2) -- (2,4);
        \draw[middlearrow={>}] (2,4) -- (4,4);
        \draw[middlearrow={>}] (4,4) -- (4,2);
        \draw[middlearrow={>}] (4,2) -- (2,2);
         \draw[latearrow={<}] (1.8,1.8) -- (4.2,1.8);
          \draw[latearrow={<}] (1.8,1.8) -- (1.8,4.2);
         \draw[earlyarrow={>}] (1.8,1.8) -- (1.8,4.2);
          \draw[earlyarrow={>}] (1.8,1.8) -- (4.2,1.8);
          \draw[glatearrow={<}] (4.2,4.2) -- (4.2,1.8);
          \draw[glatearrow={<}] (4.2,4.2) -- (1.8,4.2);
         \draw[gearlyarrow={>}] (4.2,4.2) -- (1.8,4.2);
          \draw[gearlyarrow={>}] (4.2,4.2) -- (4.2,1.8);
  \end{tikzpicture}
  \caption{ A visual explanation of condition \eqref{eq7}: The characteristic associated with the face $\bff_x$ (red) has to be larger than the product of the characteristics associated with its left and lower side (blue) ,  and smaller than the product of the characteristics associated with its upper and right side (yellow)}
\end{figure}
\vspace{3mm}
\begin{remark}
The function $t \mapsto -\log(4 \lambda^2 t(1-t))$ plays the same role as $\gamma(t)$ in the diffusion case, and it features the same stylized fact we observed for $\gamma(t)$.
\end{remark}
\begin{remark}
\begin{enumerate}[(i)]
\item
One nice aspect of \eqref{eq8} and \eqref{eq7} is that they are local conditions, that is, for a given $\bff_x$ they depend only on the closed walks of length two that intersect $\bff$. 
\item The fact that $j$ fulfills the hypothesis of the Theorem does \textit{not} imply that $j (\zz)\leq \lambda$ on every arc of the lattice. This means that there exist CSRW whose tails are heavier than the simple random walk, but the tails of their bridges are lighter than those of the bridge of the simple random walk.
\item These conditions are easy to check and there are many jump intensities satisfying them: indeed we show in Lemma \ref{faceexistence} that for any $\varphi: \cE \cup \cF \rightarrow \R_{+}$ there exist at least one intensity $j$ such that $\Phi_j(\cc) = \varphi(\cc)$ over $\cE \cup \cF$.  
\item In the proof of Theorem it is seen how condition \eqref{eq7} makes sure that among the simple closed walks with the same perimeter, the ones with smallest area are those which have the largest value of $\Phi_j(\cdot)$.
\end{enumerate}
\end{remark}

The idea of the proof of Theorem \ref{squarelatticebound} is that the local conditions we impose on the faces ensure that for any closed walk $\Phi_j(\cc)$ can be controlled in terms of $\lambda^{\ell(\cc)}$.  We then use a modification of Girsanov's theorem for bridges, which gives us a form of the density in terms of the reciprocal closed walk characteristics, and conclude that such density has a global upper bound on path space. It is likely that one can relax \eqref{eq8} \eqref{eq7}, by imposing them only in the limit when $\| x \| \uparrow +\infty$.To simplify the presentation and the proofs, we did not consider this case.

\subsection{General graphs}

Here, we consider a graph $(\cX,\to)$ satisfying Assumption \ref{as-03} below and a continuous time random walk $\bbP^x$ on $(\cX,\to)$ with intensity $j$. 
 Our aim is to prove a result similar to Theorem \ref{squarelatticebound}. 
As the notion of faces does not exist for general graphs, we work with its natural substitute: the \textit{basis of closed walks}. This notion is a slight generalization of the notion of cycle basis for an undirected graph, for which we refer to \cite[Sec. 2.6]{bondy2008graph}. 
 \subsubsection*{Trees and basis of the closed walks}
 Prior to the definition, let us recall some terminology about graphs.
A subgraph of $(\cX,\to)$ is a graph on $\cX$ whose arc set in included in the arc set $\cA$ of $(\cX,\to)$. We say that two subgraphs intersect if their arc sets do so, and we say that one is included in the other if their arc sets are so. Let us recall that for a given vertex $z\in \cX$, its outer degree is $\mathbf{deg}(z):=| \{ z'': (z \to z'') \in \cA\}| $ is the outer degree at $z$. As in the previous subsection, the set of closed walks of length two is denoted $\cE$. Figure \ref{grafogen} helps in understanding the next definition. 
\begin{mydef}[Tree and basis of closed walks]\label{defs-01}
Let $(\cX, \to)$ be a graph fulfilling Assumption \ref{as-03}.
 \begin{enumerate}[(a)]
\item We call \emph{tree} a symmetric connected subgraph $\mathcal{T}$ of $(\cX,\to)$ which spans\footnote{i.e. it connects all vertices of $(\cX,\to)$} $\cX$ and does not have closed walks of length at least three.\footnote{closed walk of length two are allowed, as the graph is symmetric}. 
\item For a tree $\cT$, $\ces$ is the the set of closed walks of length two which does not intersect $\cT$.

\be\label{e61}
\ces = \{ \bfe \in \cE : \bfe \cap \cT = \emptyset \}
\ee

\item For any $(x \to y) \in \cA \setminus \cT$ we denote $\cc_{x \to y}$ 
the closed walk obtained by concatenating $(x \to y)$ with the only simple directed walk from $y$ to $x$ in $\mathcal{T}$. 
\item Let $\mathcal{T}$ be a tree. A $\mathcal{T}$-\emph{basis of  the closed walks} of $(\cX, \to)$ is any subset $\mathcal{C}$ of closed walks of the form: 
\bes
\mathcal{C} = \mathcal{C}^*\cup \mathcal{E}
\ees
 where $\mathcal{C}^*$ is obtained by choosing for any $\bfe=(x \to y \to x) \in \ces $ exactly one among $\cc_{x \to y}$ and $\cc_{y \to x}$. We denote the chosen element by $\cc_{\bfe}$.
\end{enumerate}
\end{mydef}

\GraphInit[vstyle = Shade]
\tikzset{
  LabelStyle/.style = { rectangle, rounded corners, draw,
                        minimum width = 2em, fill = yellow!50,
                        text =black, font = \bfseries },
  VertexStyle/.append style = { inner sep=5pt,
                                font = \Large\bfseries, ball color = lightgray},
  EdgeStyle/.append style = {->, bend left=10, double=blue}}
\thispagestyle{empty}

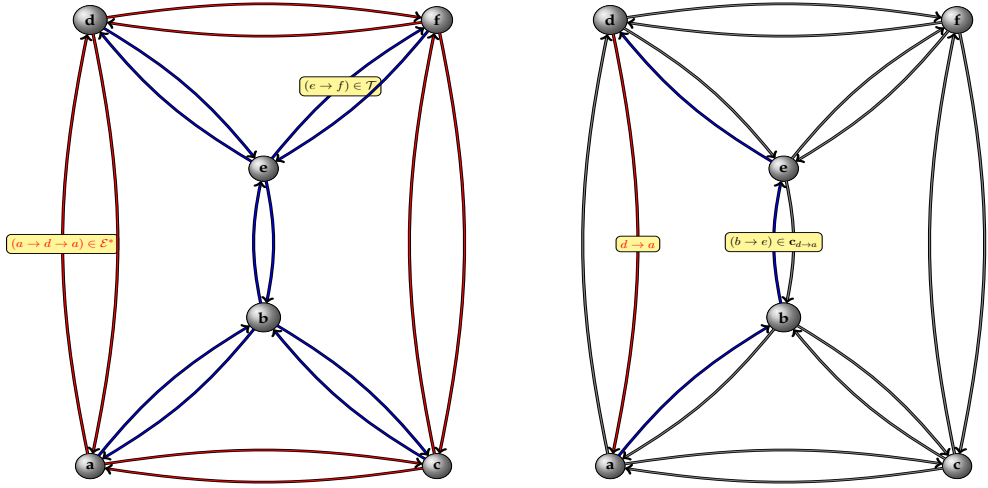
\begin{figure}[h]
\resizebox{13cm}{6.5cm}{\begin{tikzpicture}
\begin{scope}
  \SetGraphUnit{5}
  \Vertex{b}
  \SOWE(b){a}
  \SOEA(b){c}
  \NO(b){e}
  \NOWE(e){d}
  \NOEA(e){f}
    \Edge[](a)(b)
  \Edge[](b)(c)
  \Edge[](c)(b)
  \Edge[](b)(a)
  \Edge[](e)(d)
  \Edge[](d)(e)
  \Edge[label = $(e \to f) \in \mathcal{T}$ ](e)(f)
  \Edge[](f)(e)
  \Edge[](b)(e)
  \Edge[](e)(b)
  \tikzset{EdgeStyle/.append style = {double= red}}
  \tikzset{LabelStyle/.append style = {text = red}} 
  \Edge[](a)(c)
  \Edge[](c)(a)
  \Edge[](d)(f)
  \Edge[](f)(d)
   \Edge[label = $(a \to d \to a) \in \mathcal{E}^*$ ](a)(d)
  \Edge[](d)(a)
  \Edge[](c)(f)
  \Edge[](f)(c)
  \end{scope}
  
  \begin{scope}[xshift =15cm]
  \SetGraphUnit{5}
  \Vertex{b}
  \SOWE(b){a}
  \SOEA(b){c}
  \NO(b){e}
  \NOWE(e){d}
  \NOEA(e){f}
 \tikzset{EdgeStyle/.append style = {double= gray}}
 \Edge[ ](a)(d)
  \Edge[](b)(c)
  \Edge[](c)(b)
  \Edge[](b)(a)
  \Edge[](d)(e)
  \Edge[](e)(f)
  \Edge[](f)(e)
  \Edge[](e)(b)
\Edge[](a)(c)
  \Edge[](c)(a)
  \Edge[](d)(f)
  \Edge[](f)(d)
  \Edge[](c)(f)
  \Edge[](f)(c)  
  \tikzset{EdgeStyle/.append style = {double= red}}
  \tikzset{LabelStyle/.append style = {text = red}} 
  \Edge[label = $ d \to a $](d)(a) 
\tikzset{EdgeStyle/.append style = {double= blue}}
  \tikzset{LabelStyle/.append style = {text = black}} 
    \Edge[](a)(b) 
      \Edge[label = $ (b \to e) \in \mathbf{c}_{d \to a}$](b)(e)
        \Edge[](e)(d)
  \end{scope}
  \end{tikzpicture}}
\caption{Left: The blue arcs form a tree $\cT$: each pair of red arcs forms an element of $\cE^*$. Right: The closed walk $\mathbf{c}_{d \to a}$ is obtained by concatenating $(d \to a )$ with the unique simple walk in $\cT$ from $a$ to $d$ (blue).}
\label{grafogen}  
  \end{figure}

 Theorem \ref{treebound} gives a condition to control  the tails of $ d(X_t,x)$ under $\bbP^{xx}$.
\begin{theorem}\label{treebound}
Let $(\cX,\to)$ be a directed graph satisfying Assumption \eqref{as-03}, $1/\delta$ be its maximum outer degree. Let $j:\cA \rightarrow \R_+$ be the intensity of a CSRW $\bbP$ satisfying Assumption \ref{as-01}. If, for some tree $\mathcal{T}$ and a $\mathcal{T}$-based basis for the closed walks $\cC$:
\be\label{eq1}
\forall \bfe \in \cE,\quad \Phi_j(\bfe) \leq  (\lambda\delta)^{2}
\ee
 \be\label{eq2}
\forall \bfe \in \cE^*, \quad {(\lambda \delta)}^{1 - \ell(\bfe) }  \Phi_j({\bfe}) \leq  \Phi_j({\mathbf{c}_{\bfe}}) \leq {(\lambda \delta)}^{ \ell(\bfe) -1 } \prod_{ \stackrel{\bfe' \in \cE, \bfe' \neq \bfe } {\bfe' \cap \cc_{\bfe} \neq \emptyset}}   \Phi_j({\bfe'}) 
\ee
Then for any $x \in \cX$ and any $t \in \Int$, $R>0$:
\be\label{e29}
\log \bbP^{xx} \left ( d(X_t,x) \geq R \right) \leq -2R\log R +  [2 + 2 \log (\lambda t(1-t) ) + 3 \log(\delta  - 1) ]R + o(R) 
\ee
\end{theorem}
The proof of the Theorem is divided into two steps. In a first step one shows that for some constant $c$, $\bbP^{xx}(d(X_t,x) \geq R) \leq c \, {}\SRW^{xx}_{\lambda}(d(X_t,x) \geq R)$, where $\SRW^x_{\lambda}$ is the CSRW defined by:
\be\label{e93}
j(\zz) = \frac{\lambda}{\mathbf{deg}(z)}, \quad \forall \zz \in \cA.
\ee
The second step consists in estimating $\SRW^{xx}_{\lambda}(d(X_t,x) \geq R)$ with the right hand side of \eqref{e29}. Clearly, due to the fact that  $(\cX,\to)$ has no specific structure, the estimate we obtain is less precise than the on of Theorem \ref{squarelatticebound}. However, it displays the same type of decay for the tails: a leading term of order $-R \log R$ and a correction term of order $R$.

\begin{remark} We show in Lemma \ref{constspeedconstrlemma} that to any $\varphi:  \cC \rightarrow \R_{+}$ we can associate a CSRW whose reciprocal characteristics coincide with $\varphi$ over $\cC$. This shows that the conditions \eqref{eq1} and \eqref{eq2} are fulfilled by large class of Markov jump intensities. It can be seen that there exist no tree of the square lattice such that a cycle basis associated with it coincides with the faces of the lattice. Therefore Theorem \ref{squarelatticebound} is not implied by Theorem \ref{treebound}. 
\end{remark}
\vspace{0.5cm}


\section{Proof of the main results}

\subsection*{Proof of Theorem \ref{accordeon} }
\subsubsection*{Preliminaries}
We define $p^x_t(z)$ as the density of the marginal $\bbP^x_t$,  and $p^{xy}_t(z) $ as the density of $\bbP^{xy}_t$. 
Clearly, if $U$ does not depend on time, we have the relation:
\bes
p^{xy}_t(z) = \frac{p^x_t(z)p^{z}_{1-t}(y) }{p^x_1(y)}
 \ees
 $^{\alpha}p^x_t(\cdot)$ and $^{\alpha}p^{xy}_t(\cdot)$ are defined accordingly.
As Ornstein Uhlenbeck processes are Gaussian processes, for any finite set $I = \{0=t_0, t_1,t_2,..,t_{l}\} \subseteq [0,1]$
there exist a positive definite quadratic form $\qform$ over $\R^{d \times (l+1)}$ such that $\forall A \subseteq \R^{d \times l }$ and $x \in \R^d$:
\bea\label{eq10}
  ^{\alpha}\bbP^{x}(X_{I} \in A) &=& \int_{A} \exp \big(- \qform( x,x^1,..,x^l ) \big) dx^1..dx^l \\
\nonumber &=& \int_{A} \left[ p^{x}_{t_1} (x^1) \prod_{j=2}^{l} p^{x^{j-1}}_{\Delta t_j}(x^j)  \right]  dx^1..dx^l
\eea
where we set $\Delta t_j := t_j - t_{j-1}$.
Using the transition density of the Ornstein Uhlenbeck process (see e.g. \cite[Section 5.6]{KarShreve}), we can write down the explicit expression of $\qform$:
\bes
\qform(x^0,x^1,..,x^l) =  \prod_{j=1}^{l} \sqrt{ \frac{\alpha}{\pi(1- e^{-2\alpha \Delta t_j} ) }  } \exp \left( -\frac{\alpha}{ (1- e^{-2\alpha \Delta t_j }) } (x_j -e^{-\alpha \Delta t_j} x_{j-1} )^2\right) 
\ees
where we set $t_0=0$. In particular, we will be interested in the case when $I$  is the set $\Pi_m$ defined as:
\be\label{eq44}
\Pi_m = \{0, 1/m, .. ,  (m-1)/m , 1 \}
\ee
For $t\in \Int$, we define
\be\label{j(t)}
 j(t) = \max\{ j: j/m < t  \}, \quad \Pi^{ < t}_m =\{0, 1/m,..,j(t)/m,t \}
\ee
We can now prove Theorem \ref{accordeon}.
\begin{proof}
In a first step we show that the density of $\BRG$ with respect to the Brownian Bridge $\WRG$ is given by 
\be\label{e55}
\frac{d \bbP^{xy}}{d \bbW^{xy}}= \frac{1}{Z}\exp \left(- \int_{0}^{1} \scrU (t,X_t) dt \right):= M
\ee
where $\scrU$ has been defined at \eqref{e16} and $Z$ is a normalization constant. To do this, we show that the measure  
\bes
\bbQ := M \, \bbW^{xy}
\ees
fulfills the hypothesis of the Duality formula by Roelly and Thieullen, see Theorem \ref{IBPF} in the Appendix. It can be easily verified that the regularity hypothesis are verified by $\bbQ$, because of the regularity of the transition density of the Brownian bridge and of the smoothness of $\mathscr{U}$. Moreover, $\bbQ((X_0,X_1 )=(x,y))=1$. Let us now compute the derivative $\mathcal{D}_h$ of $M$. We have:
\begin{eqnarray}\label{eq:frechetdensity}
\nonumber  \mathcal{D}_h M(X) &=& \lim_{\varepsilon \rightarrow 0} \frac{1}{Z \,\varepsilon} \left( M(X+ \varepsilon h) -M(X) \right) \\
\nonumber &=& \lim_{\varepsilon \rightarrow 0} \frac{1}{Z\, \varepsilon}
\left( \exp(-\int_0^1 \scrU(t,X_t+ \varepsilon h(t))dt \right)-\left( \exp(-\int_0^1 \scrU(t,X_t )dt \right)\\
&=& \nonumber \frac{1}{Z} \left[-\int_0^1 \nabla \scrU(t,X_t) \cdot h(t) dt \right] \exp \left( - \int_0^1 \scrU(t,X_t) dt 
\right) \\
 &=& \left[-\int_0^1 \nabla \scrU(t,X_t)\cdot h(t) dt \right] M
\end{eqnarray}
 Now let us consider any simple functionals $F$. By usingTheorem \ref{IBPF}\footnote{For the application we are going to make of the duality formula to be completely justified one shall extend its validity from the simple functionals to the differentiable functionals. A simple approximation argument, which we do not present here, takes care of that. } for the Brownian bridge $\bbW^{xy}$, Leibniz's rule and \eqref{eq:frechetdensity} we obtain:
 \beas
 \bbQ \Big( \mathcal{D}_h F  \Big)& = & \bbW^{xy} \Big( (\mathcal{D}_h F)  M\Big) \\
&=& \bbW^{xy} \Big( \mathcal{D}_h (F  M) \Big)- \bbW^{xy} \Big( F (\mathcal{D}_h M)   \Big) =\\
&=&\bbW^{xy} \Big((F  M) \int_{0}^1 \dot{h}(t) \cdot dX_t \Big)+ \bbW^{xy} \Big( (F M  ) \int_{0}^{1}\scrU(t,X_t) \cdot h(t) dt \Big) \\
&=&\bbQ \Big(F \left[ \int_{0}^1 \dot{h}(t) \cdot dX_t + \int_{0}^{1}\scrU(t,X_t) \cdot h(t) dt \right] \Big)
 \eeas
 from which \eqref{e55} follows, because of the arbitrary choice of $F$. As a by-product, we obtain that, if we choose $\alpha$ as in \eqref{e17}, we have:

\bes
 \frac{d \bbP^{xy} }{ d {}^{\alpha}\bbP^{xy}} = \exp\Big(  \int_{0}^{1}V(t,X_t) dt \Big)
\ees
 where 
 \bes
 V(t,X_t) =\frac{1}{2}\alpha^2  \| x \|^2-\scrU(t,X_t) - \log (\, Z \,)
 \ees
Note tat because of \eqref{e17}, $V(t,\cdot)$ is concave for all $t$.
The next step in the proof is to prove that $z \mapsto \frac{d\bbP^{xy}_t}{d ^{\alpha}\bbP^{xy}_{t}}(z)$ is log concave. To do this we will show that $(x,z,y)\mapsto \frac{d \bbP^{xy}_t}{d ^{\alpha}\bbP^{xy}_{t}}(z)$ is log concave, which is a slightly stronger statement. To this aim, we observe that, applying the Markov property for $^{\alpha}\bbP^{xy}$ we have:
\bea\label{e51}
\nonumber \frac{d \bbP^{xy}_t}{d ^{\alpha}\bbP^{xy}_{t}}(z) &=& \bbE_{^\alpha \bbP^{xy}} \left(  \exp \left(\int_{0}^{1} V(s,X_s)ds \right)  \big | X_t = z \right) \\
&=& \bbE_{^\alpha \bbP^{xy}} \left(  \exp \left(\int_{0}^{t} V(s,X_s)ds \right)  \big | X_t = z \right) \\
\nonumber &\times &\bbE_{^\alpha \bbP^{xy}} \left(  \exp \left(\int_{t}^{1} V(s,X_s) ds \right)  \big | X_t = z \right)
\eea
 
 We show that each factor is a log concave function of $(x,y,z)$. Let us consider the first factor. A further application of the Markov property for $\bbP^x$ gives:
 \bes
 \bbE_{^\alpha \bbP^{xy}} \left(  \exp \left(\int_{0}^{t} V(s,X_s)ds \right)  \big | X_t = z \right) = \bbE_{^\alpha \bbP^{x}} \left(  \exp \left(\int_{0}^{t} V(s,X_s)ds \right)  \big | X_t = z \right):= G(x,z)
 \ees
Consider a discretisation parameter $m \in \N$, and $\Pi_m$, $j(t)$, $\Pi^{<t}_m$ as in \eqref{eq44}, \eqref{j(t)} and define:
\beas
\mathcal{I}_m& &: \R^{d \times( j(t)+2)} \rightarrow \R \\ 
\mathcal{I}_m (x,x^1..,x^{j(t)+1}) &=& \frac{1}{m} V(0,x)+ \\
&+&\frac{1}{m} \sum_{1 \leq j \leq j(t)-1} V(j/m,x^j) \, + (t- j(t)/m) V(j(t)/m,x^{j(t)})  
\eeas
and 
\beas
G^m(x,z) &:=& ^{\alpha}\bbP^{x} \Big( \exp( \mathcal{I}_m(X_{\Pi^{<t}_m}) )  \big | X_t = z  \Big) \\
&=& ^{\alpha}\bbP^{x}_{\Pi^{<t}_m}\Big( \exp( \mathcal{I}_m(x,x^1,..,x^{j(t)},x^{j(t+1)} )  \big | x^{j(t)+1} = z  \Big) 
\eeas
Clearly, $G^m(x,z) \rightarrow G(x,z)$ pointwise. The conditional density of ${}^{\alpha}\bbP^{x}_{\Pi^{<t}_m}$ given $X_t =z$ is:
\be\label{e46}
\frac{1}{{^{\alpha}}p^x_t(z)} \times \Big[{^{\alpha}} p^{ x}_{1/m}(x_1) \left( \prod_{j=2}^{j(t)}{^{\alpha}}p^{ x^{j-1} }_{1/m}(x_j)  \right) {^{\alpha}}p^{x^j}_{t-j(t)/m}(z) \Big] 
\ee
Using \eqref{eq10} we rewrite both the numerator and the normalization factor at the denominator to obtain the following equivalent expression for the conditional density:
\bes
\frac{\exp \big(- \qmform(x, x^1,...,x^{j(t)},z)  \big)}{ \int_{\R^{d \times j(t)}} \exp \big(- \qmform(x, x^1,...,x^{j(t)},z \big) dx^1..dx^{j(t)}}
\ees
 which then gives
 \bes
G^m(x,z):=\frac{ \int_{\R^{d \times j(t)} } \exp \big( \mathcal{I}_m(x,x^1,..,x^{j(t)},z) - \qmform(x, x^1,...,x^{j(t)},z)  \big) dx^1..dx^{j(t)} }{ \int_{\R^{d \times j(t)}} \exp \big(- \qmform(x, x^1,...,x^{j(t)},z)  \big) dx^1..dx^{j(t)}}
 \ees
 By mean of the identifications
 \beas
w &\hookrightarrow & (x,x^1,..,x^{j(t)},z ) \in \R^{d \times j(t)+2}\\
v &\hookrightarrow & (x^1, .. , x^{j(t)}) \in \R^{d \times j(t)} \\
v'&\hookrightarrow & (x,z) \in \R^{d \times 2} \\
F(w)&\hookrightarrow & \exp(\mathcal{I}_m(w) )
\eeas
we can then rewrite $G^m(x,z)$ as the right hand side of \eqref{eq50}. By the hypothesis \eqref{e17} $V(t,\cdot)$ is concave for any $t \in \Int$. Hence $\mathcal{I}_m$ is concave as well. Therefore we can apply Theorem \ref{thm:logconcpres}  to conclude that $G^m(x,z)$ is log-concave for all $m$, and therefore so is the limit. This concludes the proof  that the first of the two appearing in \eqref{e51} is log concave. With the same argument we have just used, one shows that also the other factor is  log concave and therefore  $\frac{d {}^{\alpha}\bbP^{xy}_t}{ d \bbP^{xy}_t}$ is log concave. This tells us that: 
\be\label{eq53}
\inf_{z \in \RD, v \in \RD, \| v \|=1} - v \cdot \Hess(\log p^{xy}_t) (z)  \cdot v \geq \inf_{z \in \RD, v \in \RD, \| v\|=1} - v \cdot \Hess(\log {}^{\alpha}p^{xy}_t) (z)  \cdot v
\ee
The explicit expression for $^{\alpha}p^x_t(z)$  is well known, see e.g. \cite[Section 5.6]{KarShreve}:
\bes
^{\alpha}p^x_t(z) =  \sqrt{\frac{\alpha}{\pi (1 -\exp(-2\alpha t))}}\exp \left( -\frac{\alpha}{(1- e^{-2 \alpha t })}  \| z - x e^{-\alpha t} \|^2    \right)
\ees
Therefore, as a function of $z$:
\beas
^{\alpha}p^{xy}_t(z) &\propto & ^{\alpha}p^{x}_t(z) ^{\alpha}p^{z}_{1-t}(y)  \\
&\propto &   \exp \left(-\frac{\alpha}{(1- e^{-2 \alpha t })}  \| z - x e^{-\alpha t} \|^2  - \frac{\alpha}{(1- e^{-2 \alpha (1-t) })}  \| y - z e^{-\alpha (1-t)} \|^2 \right)
\eeas
It is then an easy computation to show that $\Hess(\log {}^{\alpha}p^{xy}_t ) (z)   = -{}\gamma_{\alpha}(t) \mathbf{id}$, where $\gamma_{\alpha}(t)$ had been defined at \eqref{e40}.
Using \eqref{eq53}, conclusion follows by Theorem 2.7 in \cite{Led01}.
\end{proof}

\begin{remark}\label{noway}
In \cite[Th. 6.1]{BrLieb76} log-concavity of solutions to 
\bes
\partial_t \phi(t,z) - \frac{1}{2} \Delta \phi (t,z) + V(z)\phi(t,z) = 0.
\ees
is established when $V$ is convex. Define now $\phi(t,z)$ as the second factor in \eqref{e51} and assume for simplicity that $\alpha=0$ and $V$ not to depend on time:
\bes
\phi(t,z):= \bbE_{\bbW^{xy}} \left(  \exp \left(\int_{t}^{1} V(X_s) ds \right)  \big | X_t = z \right)
\ees
Using Feynamn-Kac formula and the expression for the drift of the Brownian bridge we have that $\phi$ solves
\bes
\partial_t \phi(t,z) + \frac{1}{2} \Delta \phi (t,z) + \frac{(y-z)}{(1-t)} \nabla \phi(t,z) + V(t,z)\phi(t,z) = 0.
\ees
Log-concavity of $\phi$ when $V$ is concave is a by-product of the proof of Theorem \ref{accordeon}.
\end{remark}
\subsection*{Proof of Theorem \ref{t70} }

The main steps of the proof are the Lemmas \ref{p460} and \ref{p600}. In Lemma \ref{p460} we revisit Herbst's argument, while in Lemma \ref{p460} we construct an auxiliary measure $\pi_{\Phi}$ for which sharp concentration bounds can be obtained through MLSI.

\subsubsection*{A refined Herbst's argument} We apply the Herbst's argument to a Modified Log Sobolev Inequality, studied, among others, by Dai Pra, Paganoni, and Posta in \cite{DP02}.
In their Proposition 3.1 they show that the Poisson distribution $\mu_{\lambda}(\cdot)$ of mean $\lambda$ satisfies the following inequality:
\be\label{e314} 
\forall f >0, \quad  \bbE_{\mu_{\lambda}} \Big( f \log f \Big) - \bbE_{\mu_{\lambda}} (f) \log(\bbE_{\mu_{\lambda}} (f)) \leq \lambda \bbE_{\mu_{\lambda}} \Big(\nabla f \nabla \log f \Big)
 \ee
where $\nabla f(n)$ is the discrete gradient $f(n+1)-f(n)$.

\begin{lemma}\label{p460}
Let $\mu_{\lambda}$ satisfy \eqref{e314}. Then for any 1-Lipschitz function $f: \N \rightarrow \R$:
\be\label{e315}
\mu_{\lambda} \Big(f \geq  \bbE_{\mu_{\lambda}} \big(f \big) + R \Big) \leq \exp\left(-(R+2\lambda)\big[\log\big(1+\frac{R}{2\lambda}\big) +1\big] \right)
\ee
In particular,
\bes
\mu_{\lambda} \Big(f \geq  \bbE_{\mu_{\lambda}} \big(f \big) + R \Big) \leq \exp\left(-R \log R +[\log(2 \lambda) +1]\, R +o(R)\right).
\ees
\end{lemma}
\vspace{0.5cm}
\begin{remark}\label{Herbstrem}
 We are able to improve the concentration rate obtained in \cite[Prop. 10]{BOB98} and \cite[Cor 2.2]{wu2000new} for the Poisson distribution. For instance, in \cite{BOB98} the following deviation bound for 1-Lipschitz functions is obtained under the Poisson distribution $\mu_{\lambda}$ of parameter $\lambda$:
\be\label{e25}
\mu_{\lambda} \left( f \geq  \bbE_{\mu_{\lambda}}(f) + R \right) \leq \exp \left( -\frac{R}{4} \log \left(1+\frac{R}{2\lambda}\right)\right)  
\ee
Note that the right hand side can be rewritten as $-\frac{R}{4} \log(R) + \frac{\log (2 \lambda)}{4}R + o(R)$.
We improve \eqref{e25}  to
\be\label{e26}
\mu_{\lambda} (f \geq  \bbE_{\mu_{\lambda}}(f) + R ) \leq \exp\left(-(R+2\lambda)\log\left(1+\frac{R}{2\lambda}\right) +R \right)
\ee
In this case, the rate has the form $ \exp(-R \log R + (\log(\lambda)+1+\log(2)) R + o(R)) $. This rate is sharp in the leading term $-R \log(R)$. Indeed, if one uses the explicit form of the Laplace transform of $\mu_{\lambda}$ one gets the following deviation bound for the identity function (see e.g. Example 7.3 in \cite{Ross11}):
\be\label{e27}
\mu_{\lambda} \left( n \geq \bbE_{\mu_{\lambda}}(n) + R \right) \leq \exp\left(-R\left(\log \left(1+\frac{R}{\lambda} \right)-1\right) - \lambda \log \left(1+\frac{R}{\lambda} \right)\right)
\ee
The rate here is of the form $-R \log R +(\log(\lambda)+1)R +o(R)$. This shows that \eqref{e315} is sharp concerning the leading term, has the right dependence on $\lambda$ in the exponential correction term. Concerning the constants appearing in the exponential terms, we have $1 + \log(2)$. We do not know whether this is sharp or not. However, nothing better than $1$ is reasonable to expect  because of \eqref{e27}.
\end{remark}

\begin{proof}
Let $f$ be 1-Lipschitz. It is then standard to show that $f$ has exponential moments of all order. Therefore, all the expectations we are going to consider in the next lines are finite. 
Let us define:
\bes \varphi_{\tau}:= \bbE_{\mu_{\lambda}} \big( \exp(\tau f) \big), \quad \psi_{\tau}:= \log \bbE_{\mu_{\lambda}}\left( \exp\left(\tau f\right) \right) \ees
We apply the inequality \eqref{e314} to $\exp( \tau f)$. Note that the left hand side reads as $ \tau \partial_{\tau} \varphi_{\tau} - \varphi_{\tau} \psi_{\tau}$. 
The right hand side can be written as
\bes 
\lambda \tau \bbE_{\mu_{\lambda}}( \exp(\tau f) [\exp(\tau \nabla f)-1] \, \nabla f )
\ees
Using that$f$ is 1-Lipschitz and the elementary fact that 
for all $\tau>0$ $ \sup_{y \in [-1,1]} | y [\exp(\tau y) -1] | =\exp(\tau)-1$ we can bound the above expression by:
$$ \lambda \tau [ \exp(\tau)-1 ] \, \bbE_{\mu_{\lambda}}\big( \exp(\tau f) \big)= \lambda \tau [\exp(\tau)-1]\varphi_{\tau}$$
 We thus get the following differential inequality:
\be\label{e325}
\tau \partial_{\tau} \varphi_{\tau} - \varphi_{\tau} \psi_{\tau} \leq \lambda \tau \varphi_{\tau} (\exp(\tau)-1)  
\ee
Dividing on both sides by $\varphi_{\tau}$, and using the chain rule, it can be rewritten as a differential inequality for $\psi$:
\be\label{e326}
{\tau} \partial_{\tau} \psi_{\tau} -  \psi_{\tau} \leq \lambda {\tau}(\exp({\tau})-1) , \quad \partial_{\tau} \psi_0 = \bbE_{\mu_{\lambda}}(f),\psi_0=0
\ee
The ODE corresponding to this inequality is 
\be\label{e321}
{\tau} \partial_{\tau} h_{\tau} -  h_{\tau} =  \lambda {\tau}(\exp({\tau})-1) , \quad \partial_{\tau} h_0 = \bbE_{\mu_{\lambda}}(f), h_0=0
\ee
Note that the condition $h_0=0$ is implied by the form of the equation, and it is not an additional constraint.
\eqref{e321} admits a unique solution, given by:
\be\label{e320}
h_{\tau} = {\tau} \bbE_{\mu_{\lambda}}\big(f\big)+\lambda \tau \gamma({\tau}) 
\ee
where 
\be\label{e322}
\gamma({\tau}) = \sum_{k=1}^{+\infty} \frac{1}{k} \frac{{\tau}^k}{k!}
\ee
The fact that \eqref{e320} is the solution to \eqref{e321} can be checked directly by differentiating term by term the series defining $\gamma$ in \eqref{e322}. We claim that \be\label{e323}\forall \tau \geq 0 \quad  \psi_{\tau}\leq h_{\tau}  \ee
The proof of this claim, is postponed to the Appendix section, see  Propositon \ref{p461}.
Given \eqref{e323}, a standard argument with Markov inequality yields:

\bes
\mu_{\lambda}\big(f \geq \bbE_{\mu_{\lambda}}(f)+R \big) \leq \exp\left( \inf_{\tau \geq 0}\psi_{\tau} -\tau \bbE_{\mu_{\lambda}}(f)- \tau R \right)\leq
\exp \left(\inf_{\tau >0} \lambda \tau \gamma(\tau) -\tau R  \right)
\ees
We can bound $\gamma$ in an elementary way:

\bes
\gamma(\tau)=\sum_{k=1}^{+\infty} \frac{1}{k} \frac{{\tau}^k}{k!} \leq \frac{2}{\tau} \sum_{k=1}^{+\infty}  \frac{{\tau}^{k+1}}{(k+1)!}  = 2\frac{\exp(\tau)-\tau-1}{\tau}
\ees
and therefore:
\bes
\mu_{\lambda}(f \geq \bbE_{\mu_{\lambda}}(f) + R ) \leq \exp \left(\inf_{\tau > 0}  2 \lambda \exp(\tau) - (2\lambda + R)\tau -2\lambda \right)
\ees
Solving the optimization problem yields the conclusion.
\end{proof}

\subsubsection*{An interpolation}The idea behind the proof of Theorem \ref{t70} is to construct a measure $\pi_{\Phi}$ (see Definition \ref{d107}) which ``interpolates'' $\rho_{\Phi}$ and for which the MLSI \eqref{e314} gives sharp concentration bounds. 
\begin{mydef}\label{d107}
Let $\rho_{\Phi}$ be defined by \eqref{eq21}. We define $\pi_{\Phi} \in \mathcal{P}(\N)$ as follows:
\be\label{e330} 
\pi_{\Phi} \big( m \big)= \frac{1}{Z_{\Phi}} \rho_{\Phi} \big( n(m) \big)^{1- \alpha(m) }\rho_{\Phi} \big( n(m) +1 \big)^{ \alpha(m) }
\ee
where 
\be\label{e100}
n(m) = \lfloor m/(k+1) \rfloor , \quad \alpha(m) = m/(k+1) -n(m)
\ee
\end{mydef}

Another ingredient we shall use in the proof is the following criterion for MLSI, due to Caputo and Posta. What we make here is a summary of some of their results in Section 2 of the paper \cite{CapPos07}, adapted to our scopes. To keep track of the constants, we also use Lemma 1.2 of \cite{Led01}.  We do not reprove these results here.

\begin{lemma}[Caputo and Posta criterion for MLSI,\cite{CapPos07}]\label{l10}
Let $\pi \in \mathcal{P}(\N)$ be such that
\be\label{e56}
c(m):=\frac{\pi(m-1)}{\pi(m)} 
\ee
has the property that for some $v \in \N$, $c>0$:
\be\label{e360}
\inf_{m \geq 1} \quad c(m+v)-c(m)  > 0
\ee
and that $\sup_{m \geq 0} c(m+v)-c(m)<+\infty$.
Then the function $\tilde{c}$ defined by 
\be\label{e57}
\tilde{c}(m):= c(m) + \frac{1}{v} \sum_{i=0}^{v-1} \frac{v-i}{v}[c(m+i)+c(m-i)-2c(m)]
\ee
is uniformly increasing, that is 
\be\label{e362} \inf_{m \geq 0}\tilde{c}(m+1)-\tilde{c}(m) \geq \delta \ee 
for some $\delta>0$.  
Moreover, if we define $\tilde{\pi} \in \mathcal{P}(\N)$ by:
\be\label{e361}
\tilde{\pi}(0)= \frac{1}{\tilde{Z}}, \quad \tilde{\pi}(m)= \frac{1}{\tilde{Z}}\prod_{i=1}^{m} \frac{1}{\tilde{c}(i)} 
\ee
then $\tilde{\pi}$ is equivalent to $\pi$ in the sense that there exist $\tilde{C}$ such that: 
\be\label{e363} \varepsilon \leq \frac{\pi(m)}{\tilde{\pi}(m)} \leq \varepsilon^{-1} \ee
Finally, $\pi$ satisfies the MLSI \eqref{e314} with $ \delta^{-1} \exp(4 \varepsilon^{-1} )$ instead of $\lambda$.
\end{lemma}
Using this criterion, we derive MLSI for $\pi_{\Phi}$.
\begin{lemma}\label{p600}
The measure $\pi_{\Phi}$ satisfies the MLSI \eqref{e314} with a constant of the form $ {\Phi}^{1/(k+1)} c  $, where $c$ is a constant independent from $\Phi$.
\end{lemma}

\begin{proof}

For $\Phi \in \R_{+}$ we let $c_{\Phi}$ be defined by \eqref{e56} by replacing $\pi$ with $\pi_{\Phi}$. We define $\tilde{c}_{\Phi}$  by \eqref{e57} with the choice $v=k+1$. Moreover,we define  $\delta_{\Phi}$ as in \eqref{e362}, $\tilde{\pi}_{\Phi}$
as in \eqref{e361} and $\varepsilon_{\Phi}$ as in \eqref{e363}.
Let us prove that:
\be\label{e102}
\inf_{m \geq 1} c_{1}(m+k+1) - c_{1}(m ) >0, \quad \sup_{m \geq 1} c_1(m+k+1) -c_1(m) <+\infty. 
\ee
Equation \eqref{e120} tells that:
\be\label{hdef}
\forall n \in \N, \quad \frac{\rho_{1}(n-1)}{\rho_{1}(n)} = n \times \prod_{i=0}^{k-1}\big( k n -i \big):= h(n)
\ee
By definition of $n(m)$ and $\alpha(m)$ we have that for all $m \in \N$, $n(m+k+1) = n(m)$ and $\alpha(m+k+1) = \alpha(m)$. Therefore, by definition of $\pi_1$:
\beas
c_1(m+k+1) - c_1(m) &=&  \frac{\pi_1(m+k)}{\pi_1(m+k+1)} -\frac{\pi_1(m-1)}{\pi_1(m)} \\
&=&\frac{ \rho_1( n(m-1)+1 )^{ 1-\alpha(m-1)} \rho_1( n(m-1)+2  )^{\alpha(m-1)}} {\rho_1( n(m)+1 )^{ 1-\alpha(m)} \rho_1( n(m)+2  )^{\alpha(m)}}\\
&-&\frac{ \rho_1( n(m-1) )^{ 1-\alpha(m-1)} \rho_1( n(m-1)+1  )^{\alpha(m-1)}} {\rho_1( n(m) )^{ 1-\alpha(m)} \rho_1( n(m)+1  )^{\alpha(m)}}
\eeas
We have two cases
\begin{enumerate}
\item [\underline{$m \in (k+1)\N$}] In this case $n(m-1) = n(m) -1$ and $\alpha(m)=0,\alpha(m-1) = k/(k+1)$. Therefore:
\beas
c_1(m+k+1) - c_1(m) &=& \Big[\frac{\rho_1(n(m))}{ \rho_1(n(m)+1)}\Big]^{1/k+1} - \Big[\frac{\rho_1(n(m)-1)}{ \rho_1(n(m))}\Big]^{1/k+1} \\
									&=& h^{1/(k+1)}(n(m)+1) - h^{1/(k+1)}(n(m))
\eeas
where the function $x \mapsto h(x)$ has been defined in \eqref{hdef}.
\item[\underline{$m \notin (k+1)\N$}]  In this case $n(m-1) = n(m) $ and $\alpha(m)=\alpha(m-1) +1/(k+1)$. Therefore:
\beas
c_1(m+k+1) - c_1(m) &=& \Big[\frac{\rho_1(n(m)+1)}{ \rho_1(n(m)+2)}\Big]^{1/k+1} - \Big[\frac{\rho_1(n(m))}{ \rho_1(n(m)+1)}\Big]^{1/k+1} \\
									&=& h^{1/(k+1)}(n(m)+2) - h^{1/(k+1)}(n(m)+1)
\eeas
\end{enumerate}
It can be checked with a direct computation that $h$ is strictly increasing and $\lim_{x \rightarrow + \infty} \partial_x h^{1/k+1} (x) = k^{k/(k+1)}$. Using this fact in the two expressions above yields \eqref{e102}.
We are then entitled to apply Lemma \ref{l10} which tells that $\tilde{\pi}_{1}$ satisfies the MLSI \eqref{e314} with a positive constant $\delta^{-1}_1$, and $\pi_{1}$
satisfies the MLSI with constant $\delta_1^{-1} \exp(4 \varepsilon^{-1}_1)$. Let now consider $\Phi \neq 1$.
It is an elementary observation to show that $c_{\Phi}(m) = \Phi^{-1/(k+1)}c_1(m)$. This means that (see Definition \ref{d107}):
\bes
\pi_{\Phi}(m) =  \Big[ \sum_{m=0}^{+\infty}\Phi^{-m/(k+1)} \pi_1(m)\Big]^{-1} \Phi^{-m/(k+1)} \pi_1(m)
\ees
Moreover, by construction, (see \eqref{e57}) we also have that $\tilde{c}_{\Phi}= \Phi^{-1/(k+1)}c_1$. This implies that $\delta_{\Phi} = \Phi^{1/(k+1)} \delta_1$ and that
\bes
\tilde{\pi}_{\Phi}(m) =  \Big[ \sum_{m=0}^{+\infty}\Phi^{-m/(k+1)} \tilde{\pi}_1(m)\Big]^{-1} \Phi^{-m/(k+1)} \tilde{\pi}_1(m)
\ees
It is then easy to see that, using the two expressions for $\pi_{\Phi}$ and $\tilde{\pi}_{\Phi}$ we have just derived that $\varepsilon_{\Phi} \geq \varepsilon^2_1$.
Another application od Lemma \ref{l10} gives that $\pi_{\Phi}$ satisfies the MLSI with constant $\Phi^{-1/(k+1)} \delta^{-1}_1 \exp(4 \varepsilon_{1}^{-2} ) $.
\end{proof}
We can finally prove Theorem \ref{t70}.
\begin{proof}\textit{of Theorem \ref{t70}}
Consider $f: \N \rightarrow \R$ which is 1-Lipschitz. Then define $g:\N \rightarrow \R$ by: 
\be\label{e380}
 g(m):= (1-\alpha(m) )f(n(m)) + \alpha(m) f(n(m) +1)  \ee
where $n(m),\alpha(m)$ have been defined at \eqref{e100}. It is immediate to verify that $g$ is $1/(k+1)$-Lipschitz.
Because of Lemma \ref{p600} there exists $c$ independent from $\Phi$ such that $\pi_{\Phi}$ satisfies MLSI \eqref{e314} with constant $c \, \Phi^{1/(k+1)}$. We define $M:= \Phi + \frac{\Phi^{1/K+1} }{k+1} $/ Using the concentration bound from Lemma \ref{p460} on $(k+1)g$ we get that for any $R>M$:
\beas
&{}&\pi_{\Phi} \Big( \{ m: g(m) \geq  \bbE_{\pi_{\Phi}}(g)-M+R \} \Big)\\
&\leq & \exp\Big( -(k+1)(R-M)\log(R-M) +[c+\log \Phi ](R-M)+ o(R) \Big)  \\
&=& \exp\Big( -(k+1)R\log(R) +[c+\log \Phi ]R + o(R) \Big) 
\eeas
where to obtain the last inequality we used the fact that the difference $(R-M)\log(R-M)-R\log(R)$ is a function in the class $o(R)$.
It is proven in Lemma \ref{ll} (see Appendix) that $\bbE_{\pi_{\Phi}}(g)-M \leq \bbE_{\pi}(f)$. This implies that
$\pi_{\Phi} \big( \{ m: g(m) \geq  \bbE_{\pi_{\Phi}}(g)-M+R \} \big) \geq \pi_{\Phi} \big( \{ m: g(m) \geq  \bbE_{\rho_{\Phi}}(\rho) +R \} \big) $.
Finally we observe that:
\begin{eqnarray*}
&{}&\pi_{\Phi} \big( \{ m: g(m) \geq  \bbE_{\rho_{\Phi}}(f) +R \} \big)\\ 
  &\geq &  \pi_{\Phi} \big( \{ m: g(m) \geq  \bbE_{\rho_{\Phi}}(f) +R, m \in (k+1)\N \} \big)\\  
  &= &  \frac{1}{Z_{\Phi}} \rho_{\Phi} \big( \{ n: f(n) \geq  \bbE_{\rho_{\Phi}}(f) +R \} \big)\\ 
  &\geq &  \frac{1}{k+1}\rho_{\Phi} \big( \{ n: f(n) \geq  \bbE_{\rho_{\Phi}}(f) +R \} \big)
\end{eqnarray*}
where we used the the arithmetic geometric mean inequality to show that:
 $$Z_{\Phi} = \sum_{m=0}^{+\infty} \rho_{\Phi}(n(m))^{1-\alpha(m)} \rho_{\Phi} (n(m)+1)^{\alpha(m)} \leq k+1 .$$ 
 Summing up we have:
 \beas
 \rho_{\Phi} \big( \{ n: f(n) \geq  \bbE_{\rho_{\Phi}}(f) +R \} \big) &\leq& (k+1)\pi_{\Phi} \big( \{ m: g(m) \geq  \bbE_{\rho_{\Phi}}(f) +R \} \big) \\
 &\leq& (k+1)\exp\Big( -(k+1)R\log(R) +[c+\log \Phi ]R + o(R) \Big) 
 \eeas
 The proof of the Theorem is now concluded
\end{proof}
\subsection{Proof of Theorem \ref{squarelatticebound} and \ref{treebound} }
\subsubsection*{Preliminaries}
Let us specify the assumptions on the  jump intensity.
\begin{assumption}\label{as-01} \ 
The jump intensity $j:\mathcal{A} \to \R_{+} $ verifies the following requirements.
\begin{enumerate}[(1)]
\item
It has constant speed: there exists $v >0 $ such that
\begin{equation}\label{eq-11}
\forall  z \in \cX, \quad  v=\sum_{z':\zz }j(\zz):= \jj(z)  .
\end{equation}
\item It is everywhere positive: $j(\zz)>0$ for all $\zz \in \cA$.
\end{enumerate}
\end{assumption}

Here is some vocabulary about graphs.

\begin{mydef}\label{def-02}
Let $ \cA\subset\cX^2$ specify a directed graph $(\cX,\to)$ on $\cX$ satisfying Assumption \ref{as-03}.
\begin{enumerate}[(a)]
\item The distance $d(z,z')$ between two vertices $z$ and $z'$  is the length of the shortest walk joining $z$ with $z'$. Due to point (1) of Assumption \ref{as-03}, $d$ is symmetric. 
\item If $\bw=(x_{0} \to x_1 \to .. \to x_n)$ is a walk, then $\bw^*$ is the walk obtained by reverting the orientation of all arcs:
\begin{equation}\label{eq-60}
\bw^*:=(x_n \to x_{n-1} \to .. \to x_{0})
\end{equation}
\item
A closed walk $\cc=(x_0\to x_1 \to \cdots\rightarrow x_n=x_0)$ is said to be \textit{simple} if the cardinal  of the visited vertices $ \left\{x_0,x_1,\dots, x _{ n-1}\right\} $ is equal to the length $n$ of the walk. This means that a simple closed walk cannot be decomposed into several closed walks.
A non-closed walk $\bw=(x_0 \to x_1 \to x_2 \to ... \to x_n \neq x_0 )$ is said to be \textit{simple} if the cardinal  of the visited vertices $ \left\{x_0,x_1,\dots, x _{ n}\right\} $ is equal to the length $n+1$.
\end{enumerate}
\end{mydef}
\subsubsection*{Proof of Theorem \ref{squarelatticebound}}
 The proof of Theorem \ref{squarelatticebound} is based on the following Lemma, which ensures that we can control $\Phi_j(\cc)$ in terms of $\lambda^{\ell(\cc)}$. To ease the notation, we write $\Phi(\cdot)$ instead of $\Phi_j(\cdot)$.

\begin{lemma}\label{squarepatch}
Let $j$ be as in the hypothesis of Theorem \ref{squarelatticebound}. Then for any closed walk $\cc$,  $\Phi_j(\cc) \leq \lambda^{\ell(\cc)}$.
\end{lemma}
\begin{proof}
We observe that it is sufficient to consider the case when $\cc$ is simple . Simple closed walks have an orientation, which is unique, and it can be either clockwise or counterclockwise. The interior of a closed walk is then also well defined and we call area the number of squares in the interior of $\cc$. The proof is by induction on the area of the closed walk.

\begin{enumerate}
\item[] \underline{Base step} If the area of $\cc$ is zero and $\cc$ is simple, then $\cc$ is a walk of length two, i.e. $\cc \in \cE$. The conclusion then follows by \eqref{eq8}.
\item[]\underline{Inductive step}
Consider the minimum in the lexicographic order of the vertices of $\cc$.
W.l.o.g. such vertex can be chosen to be $x_1$. By construction then, either $(x_0,x_2) = (x_1+ \mathbf{e}_1,x_1 + \mathbf{e}_2 )$ or
$ (x_0,x_2) =(x_1+ \mathbf{e}_2,x_1 + \mathbf{e}_1 )$, see Figure \ref{rett}.

\begin{itemize}
\item[(a)] \underline{$ (x_0,x_2) = (x_1+ \mathbf{e}_1,x_1 + \mathbf{e}_2 ).$} We define $z,$ $\cc_{x_2 \to x_0}$ and $\bp$ by:  
 \be\label{eq70}
 z: = x_2+v_1=x_0+v_2, \quad \cc : = ( x_0 \to x_1 \to x_2 \to \cc_{x_2 \to x_0}), \quad \bp : = (x_0 \to z \to x_2)
\ee
 We also define $\tilde{\cc}$ by concatenating $\bp$ and $\cc_{x_2 \to x_0}$ (see Figure \ref{corner}):
 \bes
 \tilde{\cc} := (\bp \to \cc_{x_2 \to x_0} )
 \ees
 We then have, recalling that $\bp^*$ is obtained by reversing $\bp$ (see Definition \ref{defs-01}):
 \beas
 \Phi(\cc) &=& j(x_0 \to x_1) j(x_1 \to x_2) \Phi(\cc_{x_2 \to x_0} ) \\
&=& \frac{  j(x_0 \to x_1) j(x_1 \to x_2)}{\Phi(\bp)}  \Phi(\bp) \Phi(\cc_{x_2 \to x_0})\\
&=&\frac{  j(x_0 \to x_1) j(x_1 \to x_2)}{ \Phi(\bp) }   \Phi(\tilde{\cc}) \\
&=&
\frac{  j(x_0 \to x_1) j(x_1 \to x_2) \Phi(\bp^*)} {\Phi(\bp^*) \Phi(\bp) } \Phi(\tilde{c})   \\
&=&
\frac{  \Phi(\mathbf{f}_{x_1} ) }{\Phi(\bfe_{x+v_2,1}) \Phi(\bfe_{x + v_1,2})} \Phi(\tilde{\cc}) 
 \eeas
 By \eqref{eq8}, $\frac{\Phi(\mathbf{f}_{x_1} ) }{\Phi(\bfe_{x_1+v_2,1}) \Phi(\bfe_{x_1 + v_1,2})}\leq 1$. Since ${\ell(\tilde{\cc})} ={\ell(\cc)}$, we would be done if we could show that $\Phi(\tilde{\cc}) \leq \lambda^{\ell(\cc)}$. We have  two cases:
\begin{itemize}
\item[(a.1)] \underline{$z$ was not touched by $\cc.$}
In this situation, $\tilde{\cc}$ is a simple closed walk. By construction, $\tilde{\cc}$ 
lies in the interior of $\cc$. Moreover
$\mathbf{f}_{x_1}$ belongs to the interior of $\cc$ but does not belong to the interior of $\tilde{\cc}$. Therefore, we can use the inductive hypothesis and obtain that $ \Phi(\tilde{\cc}) \leq \lambda^{\ell(\tilde{\cc})}$, which is the desired result.
\item[(a.2)] \underline{$z$ was touched by $\cc.$}
In this case $z = x_{j}$ for some $j \geq 3$.  We observe that we can write $\tilde{\cc}=(\tilde{\cc}_1 \to \tilde{\cc}_2)$ with $\tilde{\cc}_1=(x_2 \to .. \to x_{j}=z \to x_2 )$
and $\tilde{\cc}_2 = ( x_j=z \to x_{j+1} .. \to x_0 \to z)$ and that both $\tilde{\cc}_1$ and $\tilde{\cc}_2$ are simple closed walks which lie in the interior of $\cc$ and have disjoint interiors, see Figure \ref{fig:a2} . Moreover, since none of the walks has $\mathbf{f}_{x_1}$ in its interior, by inductive hypothesis $\Phi(\tilde{\cc}_1) \leq \lambda^{\ell(\tilde{\cc}_1)}$ and $\Phi(\tilde{\cc}_2) \leq \lambda^{\ell(\tilde{\cc}_2)} $. But then $\Phi(\tilde{\cc}) = \Phi(\tilde{\cc}_1)\Phi(\tilde{\cc}_2) \leq  \lambda^{\ell(\tilde{c}_1) + \ell(\tilde{c}_2) } =  \lambda^{\ell(\tilde{c})  }$, which is the desired result.
\end{itemize}
\item[(b)] \underline{ $ (x_0,x_2) = (x_1+ \mathbf{e}_2,x_1 + \mathbf{e}_1 )$ }
In this case the cycle the simple walk $\cc$ is counterclockwise oriented. Let $\cc_{x_2 \to x_0}$
be defined as in \eqref{eq70} above. Moreover we define
\bes
z := x_0 + v_1= x_2 +v_2 , \quad \bp: =(x_0 \to z \to x_2)
\ees
and $\tilde{\cc}:= ( \bp \to \tilde{\cc}_{x_2 \to x_0})$.
We have:

\beas
 \Phi(\cc) &=& j(x_0 \to x_1) j(x_1 \to x_2) \Phi(\cc_{x_2 \to x_0}) \\
&=& \frac{  j(x_0 \to x_1) j(x_1 \to x_2)}{\Phi(\bp)}  \Phi(\bp) \Phi(\cc_{x_2 \to x_0}) \\
&=&\frac{  j(x_0 \to x_1) j(x_1 \to x_2)}{\Phi(\bp)}   \Phi(\tilde{\cc}) \\
&=&
\frac{  j(x_0 \to x_1) j(x_1 \to x_2)  j(x_2 \to x_1)j(x_1 \to x_0)} { j(x_0 \to z)j(z \to x_2)j(x_2 \to x_1)j(x_1 \to x_0) }   \Phi(\tilde{\cc}) \\
&=&
\frac{  \Phi(\bfe_{x_1,2}) \Phi(\bfe_{x_1,1}) }{\Phi(\mathbf{f}_{x_1}) } \Phi(\tilde{\cc}) 
 \eeas
Thanks to \eqref{eq7}, $\frac{  \Phi(\bfe_{x_1,2}) \Phi(\bfe_{x_1,1}) }{\Phi(\mathbf{f}_{x_1}) } \leq 1 $. The proof that $\Phi(\tilde{\cc})\leq \lambda^{\ell(\cc)}$ is the same as in point (a).
\end{itemize}
\end{enumerate}
\end{proof}
\newpage

\tikzset{middlearrow/.style={
        decoration={markings,
            mark= at position 0.5 with {\arrow{#1}} ,
        },
        postaction={decorate},
       ultra thick,
        red
    }
}
\tikzset{bmiddlearrow/.style={
        decoration={markings,
            mark= at position 0.5 with {\arrow{#1}} ,
        },
        postaction={decorate},
        ultra thick,
        blue
    }
}

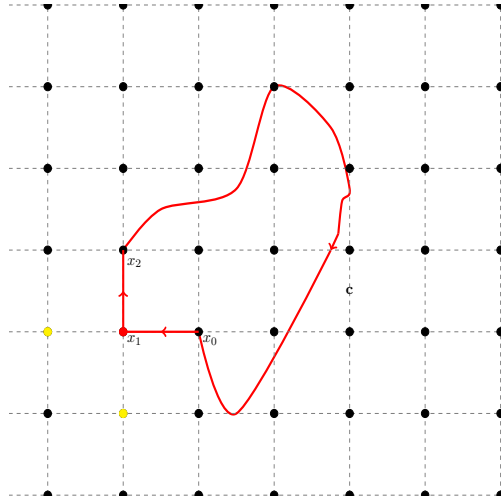
\begin{figure}[h!]
  \centering
 \resizebox{6.5cm}{6.5cm}{ \begin{tikzpicture}
    \coordinate (Origin)   at (0,0);

    \clip (-3,-2) rectangle (10cm,10cm); 
    \pgftransformcm{1}{0}{0}{1}{\pgfpoint{0cm}{0cm}}
    \coordinate (Bone) at (0,2);
    \coordinate (Btwo) at (2,-2);
    \draw [ ultra thick, red] plot [smooth  ] coordinates {(0,4) (1,5) (3,5.5) (4,8) (5.5,7) (6,5.5) (5.8,5.2)(5.7,4.4)};
       \draw [ ultra thick, red] plot [smooth]  coordinates { (5.5,4) (3,0) (2,2)};
       \draw[->,ultra thick,red] (5.7,4.4)--(5.5,4);
    \draw[style=help lines,dashed] (-14,-14) grid[step=2cm] (14,14);
    \foreach \x in {-7,-6,...,7}{
      \foreach \y in {-7,-6,...,7}{
        \node[draw,circle,inner sep=2pt,fill] at (2*\x,2*\y) {};
      }
    }
        \node[] at (0.3,1.8) {$x_1$};
          \node[] at (6,3) {$\cc$};
          \node[draw,red,circle, inner sep = 2pt,fill] at (0,2) {};
           \node[draw,yellow,circle, inner sep = 2pt,fill] at (0,0) {};
            \node[draw,yellow,circle, inner sep = 2pt,fill] at (-2,2) {};
          \node[] at (2.3,1.8) {$x_0$};
         \node[] at (0.3,3.7) {$x_2$};

        \draw[middlearrow={>}] (0,2) -- (0,4);
        \draw[middlearrow={>}] (2,2) -- (0,2);
  \end{tikzpicture}}
  \caption{A simple closed walk $\cc$ (red). $x_1$ is the minimum in the lexicographic order among the vertices visited by the closed walk. Because of that the walk cannot pass neither through the vertex left to $x_1$, nor through the vertex below $x_1$  (yellow). Therefore $\cc$ must pass through the vertices above $x_1$ and right of $x_1$. If $x_2$ is the vertex above $x_1$ the walk is clockwise oriented}
  \label{rett}
\end{figure}

\tikzset{middlearrow/.style={
        decoration={markings,
            mark= at position 0.5 with {\arrow{#1}} ,
        },
        postaction={decorate},
        thick,
        red
    }
}
\tikzset{bmiddlearrow/.style={
        decoration={markings,
            mark= at position 0.5 with {\arrow{#1}} ,
        },
        postaction={decorate},
        thick,
        blue
    }
}

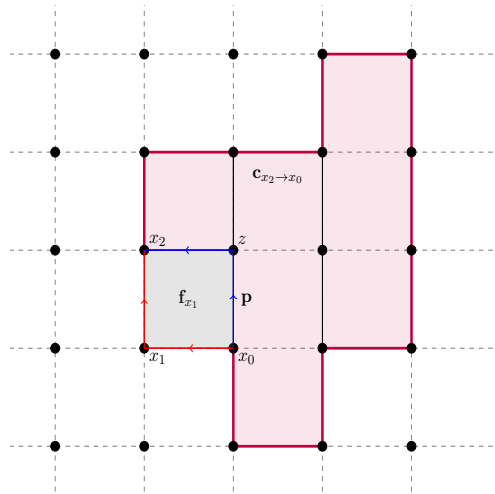
\begin{figure}[h!]
  \centering
 \resizebox{6.5cm}{6.5cm}{ \begin{tikzpicture}
    \coordinate (Origin)   at (0,0);

    \clip (-3,-1) rectangle (8cm,9cm); 
    \pgftransformcm{1}{0}{0}{1}{\pgfpoint{0cm}{0cm}}
    \coordinate (Bone) at (0,2);
    \coordinate (Btwo) at (2,-2);
    \draw[style=help lines,dashed] (-14,-14) grid[step=2cm] (14,14);
              \draw [fill=purple, fill opacity = 0.1] (0,4) rectangle (2,6);
			\draw [fill=purple, fill opacity = 0.1] (2,0) rectangle (4,6);
              \draw [fill=purple, fill opacity = 0.1] (4,2) rectangle (6,8);
              \draw[ultra thick,purple](0,4)--(0,6)--(2,6)--(4,6)--(4,8)--(6,8)--(6,6)--(6,2)--(4,2)--(4,0)--(2,0)--(2,2);    
    \foreach \x in {-7,-6,...,7}{
      \foreach \y in {-7,-6,...,7}{
        \node[draw,circle,inner sep=2pt,fill] at (2*\x,2*\y) {};
      }
    }
        \node[] at (0.3,1.8) {$x_1$};
         \node[] at (1,3) {$\mathbf{f}_{x_1}$};
                  \node[] at (2.3,3) {$\bp$};
                  \node[] at (3,5.5) {$\cc_{x_2 \to x_0}$};
          \node[] at (2.3,1.8) {$x_0$};
         \node[] at (0.3,4.2) {$x_2$};
         \node[] at (2.2,4.2) {$z$};

    \filldraw[fill=gray, fill opacity=0.2] ($(0,2)$)
        rectangle ($(2,4)$);
        \draw[middlearrow={>}] (0,2) -- (0,4);
        \draw[bmiddlearrow={<}] (0,4) -- (2,4);
        \draw[bmiddlearrow={<}] (2,4) -- (2,2);
        \draw[middlearrow={>}] (2,2) -- (0,2);
  \end{tikzpicture}}
  \caption{$\tilde{\cc}$ is constructed by cutting $(x_0 \to x_1 \to x_2)$ from $\cc$ and replacing it with $\bp= (x_0 \to z \to x_2)$ (blue). $\tilde{\cc}$ has the same perimeter  but smaller area than $\cc$}
  \label{corner}
\end{figure}
 \newpage
 
 \tikzset{middlearrow/.style={
        decoration={markings,
            mark= at position 0.5 with {\arrow{#1}} ,
        },
        postaction={decorate},
        ultra thick,
       purple
    }
}
\tikzset{bmiddlearrow/.style={
        decoration={markings,
            mark= at position 0.5 with {\arrow{#1}} ,
        },
        postaction={decorate},
        thick,
        green
    }
}

\begin{figure}[h!]
  \centering
 \resizebox{6.5cm}{6.5cm}{ \begin{tikzpicture}
    \coordinate (Origin)   at (0,0);

    \clip (-3,-1) rectangle (8cm,9cm); 
    \pgftransformcm{1}{0}{0}{1}{\pgfpoint{0cm}{0cm}}
    \coordinate (Bone) at (0,2);
    \coordinate (Btwo) at (2,-2);
    \draw[style=help lines,dashed] (-14,-14) grid[step=2cm] (14,14);
              \draw [fill=blue, fill opacity = 0.1] (0,4) rectangle (2,8);
			\draw [fill=red, fill opacity = 0.1] (2,0) rectangle (4,4);
              \draw [fill=blue, fill opacity = 0.1] (2,6) rectangle (4,8);
              \draw[ultra thick,purple](0,4)--(0,6)--(0,8)--(2,8)--(4,8)--(4,6)--(2,6)--(2,4)--(4,4)--(4,2)--(4,0)--(2,0)--(2,2)--(2,4);    
             
    \foreach \x in {-7,-6,...,7}{
      \foreach \y in {-7,-6,...,7}{
        \node[draw,circle,inner sep=2pt,fill] at (2*\x,2*\y) {};
      }
    }
        \node[] at (0.3,1.8) {$x_1$};
         \node[] at (1,3) {$\mathbf{f}_{x_1}$};
                  \node[] at (2.3,3) {$\mathbf{p}$};
          \node[] at (2.3,1.8) {$x_0$};
         \node[] at (0.3,4.2) {$x_2$};
         \node[] at (2.2,4.2) {$z$};
			 \node[] at (3,3) {$\tilde{\mathbf{c}}_2$};
         \node[] at (1,5.5) {$\tilde{\mathbf{c}}_1$};

    \filldraw[fill=gray, fill opacity=0.2] ($(0,2)$)
        rectangle ($(2,4)$);
        \draw[middlearrow={>}] (0,2) -- (0,4);
        \draw[bmiddlearrow={<}] (0,4) -- (2,4);
        \draw[bmiddlearrow={<}] (2,4) -- (2,2);
        \draw[middlearrow={>}] (2,2) -- (0,2);
  \end{tikzpicture}}
  \caption{An illustration of  case (a.2) in the proof of Lemma \ref{squarepatch}. The purple contour is $\mathbf{c}$, the green path is $\mathbf{p}$. The blue and red areas represent the interior of $\tilde{\mathbf{c}}_1$ and $\tilde{\mathbf{c}}_2$ respectively. }
  \label{fig:a2}
\end{figure}
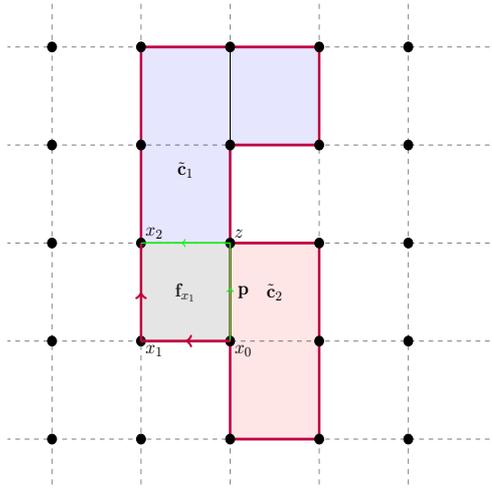


We  can now prove Theorem \ref{squarelatticebound}. Let us first state a simple Lemma we shall need, without proving it.
\begin{lemma}\label{lm:conddens}
 Let $\bbP,\bbQ$ be two probability measures on the same probability space, and let $\bbQ << \bbP$, and $M = \frac{d\bbQ}{d\bbP}$. If $A$ is an event such that $\bbQ(A) >0$ then
 \bes
\frac{ d \bbQ[\cdot \big | A]}{d \bbP[ \cdot \big | A]} = M \ind_{A} \frac{\bbP(A)}{ \bbQ(A)}
 \ees
 \end{lemma}
Thanks the last two Lemmas, the proof is then an almost straightforward application of Girsanov's theorem
\begin{proof}
Let $\bbP^x$ be a random walk of intensity $j$. We denote $\bbS^x_{\lambda}$ the random walk with constant intensity $\lambda$ started at $x$. The density of $\bbP^x$ w.r.t. to $\bbS^x_{\lambda}$ is given by (see  \cite{Jac75} or \cite{CONFphd} for a more ad-hoc version) is:
\bes\label{eq6}
\frac{d\bbP^x}{d \bbS^x_{\lambda}}=
	\exp \bigg(\sum_{i=1}^{N_1}  \log j( X_{T_{i-1}} \to X_{T_i})-\log(\lambda) -\int _{0}^1  \jj(t,X _{t^-}) + 4 \lambda \, dt \bigg)
\ees
where $N_1$ is the total number of jumps up to time $1$ and $T_i$ is the $i$-th jump time.
Since $\bbP^x$ is a CSRW, the term $\int _{0}^1\jj(t,X _{t^-})\, dt$ is constant. Moreover, if we call $\bw(X)$ the random sequence $(X_0 \to X_{T_1} \to..\to X_{T_{N_1}} )$ and use Lemma \ref{lm:conddens}  we obtain
\bes
\frac{d\bbP^{xy}}{d{}\bbS^{xx}_{\lambda} } \propto \mathbf{1}_{\{X_0=X_1=x \}}\Phi_j(\bw(X)) \lambda^{-\ell(\bw(X))}
\ees
 But then, since on the event $\{X_0=X_1=x\}$, $\bw(X)$ is a closed walk, we can apply Lemma \ref{squarepatch} to conclude that the density has a global upper bound on path space. The conclusion immediately follows from Lemma \ref{lastlemma}, which we prove in the appendix.
\end{proof}
\subsubsection*{Proof of Theorem \ref{treebound}}
Let us first specify the assumptions we make on the graph.
\begin{assumption}\label{as-03}
The directed graph $(\cX,\to)$ satisfies the following requirements:
 \begin{enumerate}[(1)]
 \item $\cA$ is symmetric:  $(x \to y )\in \mathcal{A} \Rightarrow (y \to x ) \in \mathcal{A}$.
\item It is connected: for any $x,y \in \cX^2$ there exist a directed walk from $x$ to $y$ 
\item
It is of bounded degree
\item
It has no loops, meaning that $(z \to z)\not\in \mathcal{A}$ for all $z\in\cX.$
\end{enumerate}
\end{assumption}
Let us prove the correspondent of Lemma \ref{squarepatch}.
 \begin{lemma}\label{treepatch}
Let $j$ satisfy the assumptions of Theorem \ref{treebound}. Then we have:
 \bes
 \forall \cc \in \cC, \quad \Phi_j(\cc) \leq (\lambda \delta)^{\ell(\cc)}
 \ees
 \end{lemma}
 \begin{proof}
Again, to ease the notation, we write $\Phi$ instead of $\Phi_j$. The proof goes by induction on the number of elements in $\ces$ that intersect  $\cc$. To this aim we define: 
 \bes
 \quad n(\cc) = \Big |  \left\{ \bfe \in \ces : \bfe \cap \cc \neq \emptyset \right\}  \Big |
 \ees
 
 \begin{enumerate}
\item[] \underline{Base step} If $n(\cc)=0$, then $\cc \subseteq \cT$. It is easy to see that $\cc$ can be decomposed into closed walks of length two. The conclusion then follows from \eqref{eq1}.
\item[] \underline{Inductive step} Consider any $\bfe \in \cE^*$ such that $\bfe \cap \cc \neq \emptyset$. Then there are two possible cases:
\begin{itemize}
\item[] \underline{$\big| \bfe \cap \cc \big| =2$.} In this case $\cc$ can be seen as the concatenation of $\bfe $ with two other closed walks, say $\cc_1,\cc_2$. Clearly, $n(\cc_1) ,n(\cc_2)< n(\cc)$, and therefore applying the inductive hypothesis and \eqref{eq1} we have:
\bes
\Phi(\cc) = \Phi(\cc_1) \Phi(\bfe) \Phi( \cc_2) \leq (\delta \lambda)^{\ell(\cc_1)+2 + \ell(\cc_2)} = ( \lambda \delta) ^{\ell(\cc)}
\ees
\item[]\underline{$\big| \bfe \cap \cc \big| =1$.} 
In this case, let us call $\zz$ the only arc in $\bfe \cap \cc $. By recalling the definition of  $\cc_{\bfe}$  at point (e) of Definition \ref{defs-01}, we have two subcases:
\begin{enumerate}
\item[] \underline{$\cc_{\bfe} = \cc_{\zz}$.} \
We define $\cc_{x_0 \to z}$, $\cc_{z' \to x_0}$ and $\bw_{z' \to z}$ through the following identities
\bea\label{eq5}
\cc &=&  (\cc_{x_0 \to z} \to z \to z' \to \cc_{z' \to x_0} )\\
\nonumber \mathbf{c}_{\bfe}& =&(z \to z' \to .. \to z) = (z \to z' \to \bw_{z' \to z})
\eea
  Finally, we also define $\tilde{\cc}$ as follows:
\bes
\tilde{\cc} = (\cc_{x_0,z} \to \bw^{*}_{z' \to z} \to \cc_{z,x_0})
\ees
where $\bw^{*}_{z' \to z}$  is the reversed walk (see Definition \ref{defs-01}).  Let us remark that, by definition of $\mathbf{c}_{\zz}$, we have $\bw_{z' \to z} \subseteq \cT$. But then also $\bw^*_{z' \to z} \subseteq \cT$ because $\cT$ is a symmetric graph. Therefore $ n( \tilde{\cc} )  = n(\cc)-1$. We have:
\beas
\Phi(\cc) &=& \Phi(\cc_{x_0 \to z }) j(z \to z') \Phi(\cc_{z' \to x_0} ) \\
				&=& \Phi(\cc_{x_0 \to z }) \Phi(\bw^*_{z' \to z}) \Phi(\cc_{z' \to x_0} ) \frac{ j(z \to z')}{\Phi(\bw^*_{z' \to z}) }\\
				 &=& \Phi(\tilde{\cc}) \,  \frac{ j(z \to z')}{\Phi(\bw^*_{z' \to z})}
\eeas
Using the inductive hypothesis on $\Phi(\tilde{\cc})$ we have that $\Phi({\tilde{\cc}}) \leq (\lambda\delta)^{\ell(\tilde{\cc})} = (\lambda\delta)^{+\ell(\cc) +\ell(\cc_\bfe) -1}$. If we could show that $  \frac{j(t,\zz) }{ \Phi(\bw^*_{z \to z'}) } \leq (\lambda\delta)^{-\ell(\cc_e)+1} $, then we would be done. For this aim, let us observe that by concatenating $\bw^*_{z' \to z } $ with $z ' \to z$ we obtain $ \cc_{z' \to z}$. Then:
\be\label{eq4}
  \frac{j(\zz)}{ \Phi(\bw^*_{z' \to z}) } = \frac{ j(\zz) j(z' \to z) }{ \Phi(\bw^*_{z' \to z }) j(z' \to z)  }=\frac{ \Phi(\bfe) }{ \Phi(\mathbf{c}_{z'\to z } )  }\ee

Finally, we observe that:
\bes
  \Phi(\mathbf{c}_{z' \to z } ) = \frac{1}{\Phi(\mathbf{c}_{\zz}) } \prod_{\stackrel{\bfe' \in \cE,}{ \bfe' \cap \cc_{\bfe} \neq \emptyset} } \Phi(\bfe') = \frac{1}{\Phi(\mathbf{c}_{\bfe}) } \prod_{\stackrel{\bfe' \in \cE,}{ \bfe' \cap \cc_{\bfe} \neq \emptyset} } \Phi(\bfe') 
\ees
which combined with \eqref{eq4} gives:
\bes
\frac{j(\zz)}{ \Phi(\bw^*_{z' \to z})} =  \Phi({\mathbf{c}_{\bfe}}) \left\{ \prod_{\stackrel{\bfe' \in \cE, \bfe' \neq \bfe }{ \bfe' \cap \cc_{\bfe} 
\neq \emptyset} } \Phi(\bfe')   \right\}^{-1} \ees
where we used the fact that, by construction, $\bfe$ is the only element of $\cE$ which intersects $\cc_{\bfe}$ and is not in $\cT$.
Using  the upper bound for $\mathbf{c}_{\bfe}$ in \eqref{eq2} the conclusion follows.

\item[] \underline{$\mathbf{c}_{\bfe} = \cc_{z' \to z}$}
Let $ \cc_{x_0 \to z}, \cc_{z' \to x_0} $ be defied as in \eqref{eq5}, $\bw_{z \to z'}$, and $\tilde{\cc}$ be defined by:
\bes
\mathbf{c}_{\bfe } =( z' \to z \to \bw_{\zz} ), \tilde{\cc} = ( \cc_{x_0 \to z} \to \bw_{\zz} \to \cc_{z' \to x_0})
\ees
We have:
\beas
\Phi(\cc) &=& \Phi(\cc_{x_0 \to z} ) j(\zz) \Phi(\cc_{z' \to x_0 } ) \\
				&=& \Phi(\cc_{x_0 \to z} ) \Phi( \bw_{z \to z'} )   \Phi( \cc_{z' \to x_0 } ) \frac{ j(\zz)}{\Phi( \bw_{\zz} )  }  \\
				&=& \Phi(\tilde{\cc} )  \frac{ j(\zz) }{\Phi( \bw_{\zz} )  } \\
				&=&\Phi( \tilde{\cc} )  \frac{ j(\zz) j(z' \to z ) }{\Phi( \bw_{\zz} ) j(z' \to z) }\\
				&=&\Phi(\tilde{\cc} )  \, \frac{\Phi(\bfe) }{\Phi( \mathbf{c}_{\bfe} ) }
\eeas
By construction, $n(\tilde{\cc}) = n(\cc)-1$, so we can use the inductive hypothesis together with the  the lower bound in \eqref{eq2} to obtain:
\bes
\Phi(\tilde{\cc} )  \, \frac{\Phi(\bfe) }{\Phi( \mathbf{c}_{\bfe} ) } \leq (\lambda \delta)^{\ell(\cc) + \ell(\cc_{\bfe} ) - 1}  \ (\lambda \delta)^{1-\ell(\cc_{\bfe})}  = (\lambda\delta)^{\ell(\cc)}
\ees
from which the conclusion follows.
\end{enumerate}
\end{itemize}
 \end{enumerate}
 \end{proof}
 
 The proof of Theorem \ref{treebound} can be deduced from that of Theorem \ref{squarelatticebound} by replacing $\bbS^{x}_{\lambda}$ with the random walk defined at \eqref{e93}, Lemma \ref{lastlemma} with Lemma \ref{countingest} (which we prove in the appendix), and Lemma \ref{squarepatch} with Lemma \ref{treepatch}l. Therefore, we shall not repeat it.

\subsection*{On the feasibility of \eqref{eq8},\eqref{eq7} and \eqref{eq1},\eqref{eq2} }

In this section we address the problem of how to construct jump intensities satisfying \eqref{eq8},\eqref{eq7} (resp. \eqref{eq1},\eqref{eq2}). Lemma  \ref{faceexistence} (resp. \ref{constspeedconstrlemma}) shows that for any arbitrary assignment of positive numbers $\varphi$ on $\cE \cup \cF$ (resp. $\cC$) there exists at least an intensity $j$ satisfying Assumption \ref{as-01} and such that $\Phi_j \equiv \varphi $ on $\cE \cup \cF$ (resp. $\cC$). It is then possible to construct the desired jump intensities in two steps. W.l.o.g. we restrict to the square lattice, the procedure being identical  in the case of a general graph.
\begin{itemize}
\item[\textbf{Step 1}]  Construct a positive function $\varphi$ on $\cE \cup \cF$ such that \eqref{eq8}, \eqref{eq7} hold when replacing $\Phi_j$ with $\varphi$. It is rather easy to see that this is possible.
\item[\textbf{Step 2}] Construct $j$ such that $\Phi_j = \varphi $ on $\cE \cup \cF$. The existence of such $j$ (and a way of constructing it) are given in Lemma \ref{faceexistence}
\end{itemize}
\subsubsection*{Square lattice}
Although we are interested in the lattice case, Lemma \ref{faceexistence} is easier to prove for a general planar graph. Planar graphs have a privileged set of closed walks: the \textit{faces}, which are uniquely determined once a planar representation is fixed. We choose the representation in such a way that both arcs corresponding to an element of $\cE$ on the same segment in the planar representation \footnote{This is because we do not consider the walks of length two as faces. Faces have length at least three}.
As in the case of the square lattice, the set of clockwise oriented faces of  a planar graph is denoted $\cF$.
\begin{lemma}\label{faceexistence}
Let $(\cX,\to)$ be a planar directed graph satisfying Assumption \ref{as-03}. 
Let $\varphi: \cF \cup \cE \rightarrow \R_{+}$ be bounded from above. Then there exist at least one $j: \cA \rightarrow \R_{+}$ fulfilling Assumption \ref{as-01}
and such that 
\be\label{eq124}
\forall \, \bff \in \cF, \quad \Phi_j(\bff) = \varphi(\bff), \quad \forall \, \bfe \in \cE, \quad \Phi_j(\bfe) = \varphi(\bfe)
\ee
If $\cX$ is a finite set, then $j$ is unique. If $\cX$ is  infinite, then all intensities $k:\cA \rightarrow \R_{+}$ with such properties can be written in the form
\bes
k(\zz) = \exp( \phi(z') - \phi(z) ) j(\zz)
\ees
where $h = \exp(\phi)$ is a positive solution to:
\bes
\forall z \in \cX, \sum_{z' : \zz} j(\zz) h(z') = v \, h(z)
\ees
for some constant $v>0$.
\end{lemma}

\begin{proof} 
In a first step we show the existence of a function $j:\cA \rightarrow \R_{+}$ such that \eqref{eq124} is satisfied.  The proof goes by induction on the number of arcs of $(\cX,\to)$. The base step is trivial. For the inductive step,  consider two clockwise orient  faces $\bff_1,\bff_2$ which are \textit{adjacent}. This means that there exist $\bfe_0=(x \to y \to x) \in \cE$ such that $(x \to y) \in \bff_1$ and $(y \to x) \in \bff_2$. Consider the graph $(\cX, \to_{1})$ obtained by removing  $\bfe_{0}$ from $(\cX,\to)$.  This planar graph instead of the two faces $\bff_1$ and $\bff_2$ has a single face $\bh$, which corresponds to the union of $\bff_1$ and $\bff_2$. 
On this $(\cX,\to)$ we define $\psi: \cE\setminus \bfe_0 \cup \cF \setminus \{ \bff_1,\bff_2 \} \cup \bh$ as follows:
\beas\label{e82}
\nonumber \forall \bfe \in \cE \setminus \bfe_{0} \quad \psi(\bfe) = \varphi(\bfe) \\
\forall \bff \in \cF \setminus \{ \bff_1 , \bff_2 \} \quad \psi(\bff) = \varphi(\bff) 
\eeas
\be\label{e83}
\psi(\bh) = \frac{ \varphi(\bff_1) \varphi(\bff_2)}{ \varphi( \bfe_0)}
\ee
By the inductive hypothesis there exist $j:\cA \setminus \bfe_0 \rightarrow \R_{+} $
such that \be\label{e80} \forall \quad \bfe \in \cE \setminus \bfe_0 , \quad \Phi_j(\bfe) = \psi(\bfe) \ee   and 
\be \label{e81} \forall \bff \in \cF \setminus \{ \bff_1,\bff_2\} \cup \{ \bh \}, \quad \Phi_j(\bff) = \psi(\bff).\ee 
Consider $\bff_1 =(x \to y \to x_2 \to .. \to x)$ and $\bff_2 = ( y \to x \to y_2 \to ..\to y  )$.	
We extend $j$ to $\bfe_0$ by defining:
\bea\label{e84}
\nonumber j(x \to y ) &=& \varphi(\bff_1) \big[j(y \to x_2) \prod_{i=2}^{\ell(\bff_1)-1} j(x_i \to x_{i+1})\big]^{-1} \\
j(y \to x )&=& \varphi(\bff_2) \big[j(x \to y_2) \prod_{i=2}^{\ell(\bff_2)-1} j(y_i \to y_{i+1})\big]^{-1}
\eea
We claim that $j$ as constructed here satisfies \eqref{eq124}. For $\bff \neq \bff_1,\bff_2$ and $\bfe \neq \bfe_0$, this is granted by \eqref{e80} and \eqref{e81}. Using \eqref{e84}, it is seen that $\Phi_j(\bff_1) = \varphi(\bff_1)$ and $\Phi_j(\bff_2) = \varphi(\bff_2)$. Therefore we only need to check $\bfe_0$. Using \eqref{e83}, the inductive hypothesis and what we have just proven:
\bes
\Phi_j(\bfe_0 ) = \frac{\Phi_j(\bff_1) \Phi_j(\bff_2)}{\Phi_j(\bh)} = \frac{\varphi(\bff_1) \varphi(\bff_2)}{\psi(\bh)} \stackrel{\eqref{e83}}{=} \varphi(\bfe_0)
\ees
which is the desired conclusion. This concludes the proof that an intensity $j:\cA \rightarrow \R_{+}$ satisfying \eqref{eq124} exists. To complete the proof we show that it is possible to modify $j$ in such a way that both Assumption \ref{as-01} and \eqref{eq124} are satisfied. For this purpose, we observe if $j$ is an intensity satisfying \eqref{eq124} all other intensities $k: \cA \rightarrow \R_{+}$ fulfilling \eqref{eq124} are of the form 
\bes
k(\zz) = \exp(\phi(z') - \phi(z)) j(\zz)
\ees
where $\phi : \cX \rightarrow \R$ is some potential on $\cX$. For assumption \ref{as-01} to hold, there must exist $v>0$ such that $\bar{k}(z) \equiv v$ for all $z \in \cX$. 
Let us define $h:=\exp(\phi)$. What we look for is then a pair $h,v$ such that 
\bes \forall z \in \cX, \quad \sum_{z' :\zz} j(\zz) h(z')  = v h(z), \quad h(z)>0 \, \forall \, z \in \cX \ees
Since w.l.o.g $\cX \subseteq \N$, if we define the matrix $K=(k_{m,n})_{m, \in \N}$ with $k_{m,n}:=k(m \to n)$, we can rewrite the former equation as:
\bes
K \cdot h = v h, \quad v>0, h>0
\ees
If $\cX$ is finite, the existence of a solution is ensured by the standard Perron Frobenius Theorem. The uniqueness statement is a consequence of the fact that the eigenspace of the positive eigenvalue has dimension 1. If $\cX$ is infinite and countable, we can use Corollary of Theorem 2 at page 1799 of \cite{pruitt1964eigenvalues}. We are entitled to use the Corollary because  $(\cX,\to)$ is of  bounded degree. 
\end{proof}
\subsubsection*{General graph}
\begin{lemma}\label{constspeedconstrlemma}
Let $(\cX,\to)$ be a graph fulfilling Assumption \ref{as-03}, $\mathcal{T}$ be a tree and $\cC$ be a $\cT$-basis of the closed walks. 
Let $\varphi: \cC \rightarrow \R_{+}$ be bounded from above. Then there exist $j: \cA \rightarrow \R_{+}$
such that Assumption \eqref{as-01} is satisfied
and 
\be\label{eq27}
\forall \cc \in \cC, \quad \Phi_j(\cc) = \varphi(\cc)
\ee
If $\cX$ is a finite set, then $j$ is unique. If $\cX$ is  infinite, then all other functions $k:\cA \rightarrow \R_{+}$ fulfilling Assumption \ref{as-01} and \eqref{eq27} can be written in the form
\bes
k(\zz) = \exp( \phi(z') - \phi(z) ) j(\zz) 
\ees
where $h = \exp(\phi)$ solves is a positive solution to:
\bes
\forall z \in \cX, \sum_{z' ; \zz} j(\zz) h(z') = v h(z)
\ees
for some constant $v>0$.
 \end{lemma}

Here is the proof of Lemma \ref{constspeedconstrlemma}.

\begin{proof}
We only show that we can construct $j:\cA \rightarrow \R_{+}$ such that \eqref{eq27} is satisfied. The proof that $j$ can be turned into an intensity $k$ satisfying Assumption \ref{as-01} can be done following  Lemma \ref{faceexistence} with almost no change. For any $\bfe=(x\to y \to x) \in \cE \setminus \cE^*$ (i.e. $\bfe \subseteq \cT$), we choose exactly one among $(x \to y )$ and $(y \to x)$ and set the value of $j(x \to y)$ to an arbitrary positive value. Then we set $j(y \to x) = \frac{\varphi(\bfe)}{j(x \to y)}$.
Next, for any $\bfe \in \cE^*$ we let $x \to y$ be the arc of $\bfe$ such that $\cc_{x \to y} = \cc_{\bfe}$. We observe that $\cc_{x \to y}$ can be written as $( x \to y \to \bp_{y \to x})$ for some simple walk $\bp_{y \to x}$ from $y$ to $x$ whose arcs are in $\cT$. The value of $j$ has been already set on $\bp_{y \to x}$: therefore we can then set $ j(x \to y)$ as $\frac{\varphi(\cc_{\bfe})}{ \Phi_j(\bp_{y \to x})}$. Finally we set $j$ on $y\to x$ by  $j(y \to x ):= \varphi(\bfe)/j(x \to y)$. It is then easy to check that the intensity $j$ so constructed satisfies \eqref{eq27}. 
\end{proof}

\begin{center}
{$\mathbf{Acknoledgments}$}
\end{center}
The author wishes to thank  Paolo dai Pra  and Sylvie Roelly for having introduced him to the subject, and for giving several advises during the preparation of the manuscript.  Many thanks to Christian L\'eonard, Cyril Roberto and  Max Von Renesse for insightful discussions.

\appendix
 \section{ }\label{sec:A}
 The appendix is organized as follows: we first recall the main tools used in the proof of Theorem \ref{accordeon}. Then we prove the two Lemmas \ref{p461} and \ref{ll}, which are needed in the proof of Theorem \ref{t70}. Finally, we prove Lemma \ref{lastlemma}, which is part of the proof of Theorem \ref{squarelatticebound}.
\subsubsection*{About Theorem \ref{accordeon}.}
We recall two of the main ingredients used in the proof. The first one is the integration by parts (duality) formula proved in \cite[Th.4.1]{RT05} to characterize bridges of Brownian diffusions. Here, we report a slightly simplified version of the formula, which still suffices for the scopes this paper.
\begin{theorem}[Integration by parts formula]\label{IBPF}
Let $\bbP^x$ be law of
\bes
d X_t = - \nabla U (t,X_t)dt + dB_t, \quad X_0 =x
\ees
Let $\bbQ$ be a probability measure on $C([0,1],\RD)$ satisfying the regularity hypothesis (A0),(H1),(H2) of Theorem 4.1 in \cite{RT05}. Then $\bbQ$ is the bridge $\bbP^{xy}$ if and only if $\bbQ((X_0,X_1)=(x,y))=1$ and the formula
\begin{equation*}
\mathbb{E}_{\mathbb{Q}} \Big( \mathcal{D}_{h} F \Big) =\mathbb{E}_\mathbb{Q} \left( F \int_{0}^{1} \dot{h}(t) \cdot d X_t \right) + \mathbb{E}_\mathbb{Q} \left( F \int_{0}^1 \nabla \scrU (t,X_t) \cdot h(t) dt \right)
\end{equation*}
holds for any simple functional $F$, and any direction of differentiation $h$ which is continuous, piecewise linear and satisfies the loop condition
\begin{equation*}
\quad  h(1)=h(0)=0
\end{equation*}
\end{theorem}
Let us recall that by a simple functional we mean a functional that can be written in the form $\varphi(X_{t_1},..,X_{t_k})$ for some $\mathcal{C}^{\infty}_{b}(\R^{d \times k})$ function $\varphi$ and finitely many $t_1,..,t_k$. The directional Fr\'{e}chet derivative $\mathcal{D}_h F$ of the simple functional $F$ is defined as usual:
\beas
\mathcal{D}_h F &= &\lim_{\varepsilon \rightarrow 0 }  \frac{ \varphi(X_{t_1} + \varepsilon h(t_1),..,X_{t_k}+ \varepsilon h(t_k) ) -\varphi(X_{t_1} ,..,X_{t_k})}{\varepsilon}\\
&=&  \sum_{j=1}^{k} \sum_{i=1}^d \partial_{x^j_i} \varphi(X_{t_1},..,X_{t_k}) h_i (t_j) 
\eeas 
The second is a Theorem proved in\cite{BrLieb76} that gives a quantitative version of the statement that marginalization preserves log concavity. We follow the presentation of \cite{Simon2011}.
\begin{theorem}[Preservation of strong log concavity]\label{thm:logconcpres}
Let $F: \R^{m+n} \rightarrow \R_{+}$ be log concave and let $\Sigma(\cdot)$ be a positive quadratic form on $\R^{m+n}$ . Write $w=(v,v')$, with $z \in \R^{m+n}$, $v \in \R^m$, $v' \in \R^n$. Let $F(w)$ be jointly log concave on $\R^{m+n}$ and define on $\R^n$,
\be\label{eq50}
G(v') = \frac{\int_{\R^m} F(w) \exp(-  \Sigma (w)  ) dv}{\int_{\R^m} \exp(- \Sigma( w ) ) dv }
\ee
Then $v' \mapsto G(v')$ is log concave.
\end{theorem}
For the proof we refer to \cite[Theroem 13.3, pag.204]{Simon2011} or \cite[Theorem 4.3]{BrLieb76}. 
\subsubsection*{Proof of Lemma \ref{p461}}

\begin{lemma}\label{p461}
Let $h$ be defined by \eqref{e321} and $\psi$ be as in \eqref{e326}
Then 
$$ \forall \tau >0  , \quad  \psi_{\tau} \leq h_{\tau} $$
\end{lemma}
\begin{proof}
Consider $\varepsilon>0$ and define $h^{\varepsilon}_{\tau}$ as the unique solution of 
\begin{equation}
{\tau} \partial_{\tau} h^{\varepsilon}_{\tau}- h^{\varepsilon}_{\tau} = {\tau} (\exp({\tau})-1), \quad  \partial_{\tau} h^{\varepsilon}_0 = \rho(f) + \varepsilon
\end{equation}

Then $\eta^{\varepsilon}_0:=\psi_0 - h^{\varepsilon}_0=0$ satisfies: 
\bes 
\tau \partial_{\tau}\eta^{\varepsilon}_{\tau} - \eta^{\varepsilon}_{\tau} \leq 0, \quad \partial_{\tau} \eta^{\varepsilon}_0 =- \varepsilon 
\ees
 Since $\eta^{\varepsilon}$ is continuously differentiable, we have that $T>0$, where $T$ is defined as 
\begin{equation}
T:= \inf \{\tau >0 : \partial_{\tau} \eta^{\varepsilon}_{\tau}=0\}
\end{equation}
Assume that $T<+\infty$. Then, at $T$, we have:
\be T \underbrace{\partial_{\tau} \eta^{\varepsilon}_T}_{=0} - \eta^{\varepsilon}_T \leq 0 \Rightarrow  \eta^{\varepsilon}_T \geq 0 \ee
But this is impossible since $\eta^{\varepsilon}_0=0, \partial_{\tau} \eta^{\varepsilon}_{\tau}<0$ for all $\tau < T$. Therefore $\partial_{\tau} \eta^{\varepsilon}_{\tau} <0 $ for all $\tau>0$. Since $\eta^{\varepsilon}_0=0$, we also have that $\psi^{\varepsilon}_{\tau}<0$ for all $\tau>0$. Therefore, as the choice of $\varepsilon$ was arbitrary:
\bes 
\forall \tau>0, \quad \psi_{\tau} \leq \inf_{\varepsilon >0} h^{\varepsilon}_{\tau}= h_{\tau}
\ees
\end{proof}

\subsubsection*{Proof of Lemma \ref{ll}}
\begin{lemma}\label{ll}
\bes
  \bbE_{ \pi_{\Phi} } \left( g \right)- (\Phi + \frac{1}{k+1}\Phi^{1/(k+1)}) \leq \bbE_{\rho_{\Phi}} \left( f  \right) 
\ees
\end{lemma}
\begin{proof}
By construction of $g$, see \eqref{e380} we can w.l.o.g assume that. $f(0)=g(0)=0$.  By \eqref{e120}, we have that $\rho_{\Phi}(n) \leq  \frac{\Phi}{n} \rho_{\Phi}(n-1)$ \footnote{Actually, the quotient $\rho_{\Phi}(n)/\rho_{\Phi}(n-1)$ is of the order $1/n^{k+1}$. However, here it suffices to consider $1/n$}. Therefore, using the 1-Lipschitzianity of $f$ and $f(0)=0$:
\bes
\bbE_{\rho_{\Phi}}(f) \geq - \sum_{n=1}^{+\infty} n \rho_{\Phi}(n) \geq -\Phi \sum_{n=1}^{+\infty}  \rho_{\Phi}(n-1) \geq - \Phi.
\ees
By construction, $g$ is $1/(k+1)$ Lipschitz, and w.l.o.g. $g(0)=0$. Moreover, it is easy to see from the definition of $\pi_{\Phi}$ given at \eqref{e330} that we have:
$\pi_{\Phi}(n) \leq \frac{ \Phi^{1/k+1}}{n } \pi_{\Phi}(n-1) $. Using all this:
\bes
\bbE_{\pi_{\Phi}}(g) \leq \frac{1}{k+1} \sum_{n=1}^{+ \infty} n \pi_{\Phi}(n) \leq  \frac{\Phi^{1/k+1} } {k+1} \sum_{n=1}^{+ \infty} \pi_{\Phi}(n-1) \leq  \frac{ \Phi^{1/k+1} }{k+1}
 \ees
 The proof is complete.
\end{proof}
\subsubsection*{Lemmas \ref{countingest} and \ref{lastlemma} }
\begin{lemma}\label{countingest}
Let $(\cX,\to)$ be a graph satisfying Assumption \ref{as-03} and let ${}\SRW^x_{\lambda}$ be the simple random walk defined at \eqref{e93}. Then
\bes
\log {}\SRW^{xx}_{\lambda}(d(X_t,x) \geq R) \leq  -2R\log R + R [2 + 2 \log (\lambda t(1-t) ) + 3 \log(\delta  - 1) ] + o(R) 
\ees
\end{lemma}
\begin{proof}
Let $x,y \in \cX$. We first show that for some $c_1>0$: 
\be\label{eq203}
\SRW^{x}_{\lambda}(X_t = y) \leq c_1 \frac{1}{d(x,y)!} (\delta-1)^{d(x,y)}
\ee
To this aim we define $\mathbf{W}_{k}$ as the set of walks of length $k$ which begin at $x$ and end at $y$.
We have, by conditioning on the total number of jumps up to time $t$:
\bes
{}\SRW^{x}_{\lambda}(X_t = y) = \exp(-\lambda t) \sum_{k= d(x,y)}^{+\infty} \frac{(\lambda t)^k }{k!} \sum_{\bw \in \mathbf{W}_{k} } \lambda^{-k} \Phi_j(\bw).
\ees
It is rather easy to see that $ \lambda^{-k} \Phi_j(\cc) \leq 1$. Moreover, the cardinal of $\mathbf{W}_{k} $ can be bounded above by $\delta (\delta-1)^{k-2}$. Using these two observations:
\be\label{eq200}
\SRW^{x}_{\lambda}(X_t = y) \leq  \exp(-\lambda t)\sum_{k= d(x,y)}^{+\infty}  \frac{(\lambda t)^k }{k!} \delta (\delta-1)^{k-2} 
\ee
A standard argument based on Stirling's formula shows that the sum appearing in \eqref{eq200} can be controlled with its first summand, i.e. there exist a constant $c_1$ independent from $d(x,y)$ such that: 
\be\label{eq201}
\SRW^{x}_{\lambda}(X_t = y)  \leq c_1 \frac{(\lambda t)^{d(x,y)} }{d(x,y)!} \delta (\delta-1)^{d(x,y)-2} 
\ee
which proves \eqref{eq203}. Since there cannot be more than  $(\delta-1)^{R}$ vertices at distance $R$ from $x$, we get that, using twice \eqref{eq201}:
\beas
 \SRW^{xx}(d(X_t,x) = R)=  \frac{1}{\SRW^x_{\lambda}(X_1=x)} \sum_{y: d(x,y)=R}  {}\SRW^{x}_{\lambda}(X_t = y)  {}\SRW^{y}_{\lambda}(X_{1-t} = x )\\
 \leq c_2 \frac{(\lambda^2 t(1-t))^{R} }{R!^2} (\delta-1)^{3R} 
\eeas
for some $c_2>0$. Therefore:
\bes
 \SRW^{xx}(d(X_t,x) \geq R) \leq c_2 \sum_{k=R}^{+\infty}  \frac{(\lambda^2 t(1-t))^{k} }{k!^2}  (\delta-1)^{3k}
 \ees
Using again a standard argument with Stirling formula as we did in \eqref{eq200}, we obtain:
\bes
 \SRW^{xx}(d(X_t,x) \geq R) \leq c_3  \frac{(\lambda^2 t(1-t))^{R} }{R!^2}  (\delta-1)^{3R}
\ees
for some $c_3>0$. The conclusion follows from Stirling's formula, which allows to write $\log R! = R \log R - R + o(R) $.
\end{proof}
\begin{lemma}\label{lastlemma}
Let $\bbS^{x}_{\lambda}$ be the constant speed random walk on the square lattice defined by:
$$ j(x \to x+ v_1) = j(x \to x+ v_2) \equiv \lambda $$
Then:
\be\label{e105}
\log {}\bbS^{xx}_{\lambda}( d(X_t,\bbE_{{}\bbS^{xx}_{\lambda}}(X_t) ) \geq R) = -2 R \log R + [\log(4 \lambda^2 t(1-t) ) +2  ]R + o(R) 
\ee
\end{lemma}
Lemma \ref{lastlemma} is not directly implied by Lemma \ref{countingest}.
However, one can derive its proof by going along the same lines of the proof of Lemma \ref{countingest} and use the exact computations that can be performed for the square lattice. A detailed proof is available at...

\bibliographystyle{plain}
\bibliography{/Users/giovanniconforti/Desktop/command&Ref/Ref}

\begin{thebibliography}{10}

\bibitem{bailleul2013large}
I.~Bailleul.
\newblock Large deviation principle for bridges of degenerate diffusion
  processes.
\newblock {\em Preprint arXiv, available at http://arxiv. org/pdf/1303.2854.
  pdf}, 2013.

\bibitem{bailleul2015small}
l.~Bailleul, L.~Mesnager, and J.~Norris.
\newblock Small-time fluctuations for the bridge of a sub-riemannian diffusion.
\newblock {\em arXiv preprint arXiv:1505.03464}, 2015.

\bibitem{BAKEM}
D.~Bakry and M.~{\'E}mery.
\newblock Diffusions hypercontractives.
\newblock In {\em S{\'e}minaire de Probabilit{\'e}s XIX 1983/84}, pages
  177--206. Springer, 1985.

\bibitem{baldi2002asymptotics}
P.~Baldi and L.~Caramellino.
\newblock Asymptotics of hitting probabilities for general one-dimensional
  pinned diffusions.
\newblock {\em The Annals of Applied Probability}, 12(3):1071--1095, 2002.

\bibitem{baldi2014large}
P.~Baldi, L.~Caramellino, and M.~Rossi.
\newblock Large {D}eviation asymptotics for the exit from a domain of the
  bridge of a general diffusion.
\newblock {\em arXiv preprint arXiv:1406.4649}, 2014.

\bibitem{Blee}
{I.} Benjamini and {S.} Lee.
\newblock Conditioned diffusions which are {B}rownian bridges.
\newblock {\em Journal of Theoretical Probability}, 10(3):733--736, 1997.

\bibitem{BOB98}
{S.} Bobkov and {M.} Ledoux.
\newblock On modified logarithmic {S}obolev inequalities for {B}ernoulli and
  {P}oisson measures.
\newblock {\em {J}ournal of functional analysis}, 156(2):347--365, 1998.

\bibitem{bondy2008graph}
J.~Bondy and U.~Murty.
\newblock Graph theory, volume 244 of {G}raduate {T}exts in {M}athematics,
  2008.

\bibitem{BrLieb76}
H.J. Brascamp and E.H. Lieb.
\newblock On extensions of the {B}runn-{M}inkowski and {P}r{\'e}kopa-leindler
  theorems, including inequalities for log concave functions, and with an
  application to the diffusion equation.
\newblock {\em Journal of {F}unctional {A}nalysis}, 22(4):366--389, 1976.

\bibitem{CapPos07}
{P.} Caputo and {G.} Posta.
\newblock Entropy dissipation estimates in a zero-range dynamics.
\newblock {\em Probability theory and related fields}, 139(1-2):65--87, 2007.

\bibitem{chaumont2011markovian}
L.~Chaumont and G.U. Bravo.
\newblock Markovian bridges: weak continuity and pathwise constructions.
\newblock {\em The Annals of Probability}, 39(2):609--647, 2011.

\bibitem{Chen}
{L.H.Y.} Chen.
\newblock Poisson approximation for dependent trials.
\newblock {\em The Annals of Probability}, 3(3):534--545, 1975.

\bibitem{Cl91}
{J.M.C.} Clark.
\newblock A local characterization of reciprocal diffusions.
\newblock {\em Applied Stochastic Analysis}, 5:45--59, 1991.

\bibitem{conforti2016counting}
{G.} Conforti.
\newblock Bridges of markov counting processes: quantitative estimates.
\newblock {\em to appear in Electronic Communications in Probability, available
  at http://arxiv.org/abs/1512.01180}, 2015.

\bibitem{CONFphd}
G.~Conforti.
\newblock PhD thesis, Universitaet Potsdam and University of Padova, 2015.
  Availble at
  https://publishup.uni-potsdam.de/opus4-ubp/frontdoor/index/index/docId/7823.

\bibitem{CDPR}
{G.} Conforti, {P.} Dai~Pra, and {S.} Roelly.
\newblock Reciprocal classes of jump processes.
\newblock {\em to appear in Journal of Theoretical probability, online on the
  journal website at
  http://link.springer.com/article/10.1007/s10959-015-0655-3/fulltext.html},
  2015.

\bibitem{CL15}
{G.} Conforti and {C.} L{\'e}onard.
\newblock Reciprocal classes of a random walk on graphs.
\newblock {\em Preprint arxiv:1505.01323}.

\bibitem{CLMR}
{G.} Conforti, {C.} Léonard, {R.} Murr, and {S.} Roelly.
\newblock Bridges of {M}arkov counting processes. {R}eciprocal classes and
  duality formulas.
\newblock {\em Electron. Commun. Probab.}, 20:1--12, 2015.

\bibitem{CR}
{G.} Conforti and {S.} Roelly.
\newblock Reciprocal class of random walks on an {A}belian group.
\newblock {\em To appear in Bernoulli, preprint at
  http://publishup.uni-potsdam.de/opus4-ubp/frontdoor/index/index/docId/7019}.

\bibitem{CruZa91}
A.B. Cruzeiro and J.C. Zambrini.
\newblock Malliavin calculus and {E}uclidean quantum mechanics. {I}.
  {F}unctional calculus.
\newblock {\em Journal of Functional Analysis}, 96(1):62--95, 1991.

\bibitem{DP02}
{P.} Dai~Pra, {A.M.} Paganoni, and {G.} Posta.
\newblock Entropy inequalities for unbounded spin systems.
\newblock {\em Annals of {P}robability}, 30:1959--1976, 2002.

\bibitem{dawson1990schrodinger}
D.~Dawson, L.~Gorostiza, and A.~Wakolbinger.
\newblock Schr{\"o}dinger processes and large deviations.
\newblock {\em Journal of mathematical physics}, 31(10):2385--2388, 1990.

\bibitem{Doob1957}
J.L. Doob.
\newblock Conditional {B}rownian motion and the boundary limits of harmonic
  functions.
\newblock {\em Bulletin de la Société Mathématique de France}, 85:431--458,
  1957.

\bibitem{fang1999integration}
Shizan Fang.
\newblock Integration by parts formula and logarithmic sobolev inequality on
  the path space over loop groups.
\newblock {\em The Annals of Probability}, 27(2):664--683, 1999.

\bibitem{Fitz}
P.J. Fitzsimmons.
\newblock Markov processes with identical bridges.
\newblock {\em Electron. J. Prob}, 3, 1998.

\bibitem{Jac75}
{J.} Jacod.
\newblock Multivariate point processes: predictable projection,
  {R}adon-{N}ikodym derivatives, representation of martingales.
\newblock {\em Zeitschrift für Wahrscheinlichkeitstheorie und Verwandte
  Gebiete}, 31(3):235--253, 1975.

\bibitem{Joulin2007Poisson}
A.~Joulin.
\newblock Poisson-type deviation inequalities for curved continuous-time markov
  chains.
\newblock {\em Bernoulli}, pages 782--798, 2007.

\bibitem{KarShreve}
I.~Karatzas and S.~Shreve.
\newblock {\em Brownian motion and stochastic calculus}, volume 113.
\newblock Springer Science \& Business Media, 2012.

\bibitem{Kre88}
{A. J.} Krener.
\newblock Reciprocal diffusions and stochastic differential equations of second
  order.
\newblock {\em Stochastics}, 107(4):393--422, 1988.

\bibitem{Kre97}
{A.J.} Krener.
\newblock Reciprocal diffusions in flat space.
\newblock {\em Probability Theory and Related Fields}, 107(2):243--281, 1997.

\bibitem{Led01}
{M.} Ledoux.
\newblock {\em The Concentration of Measure Phenomenon}, volume~89 of {\em
  Mathematical Surveys and Monographs}.
\newblock American Mathematical Society, 2001.

\bibitem{LK96}
{B.C.} Levy and {A.J.} Krener.
\newblock Stochastic mechanics of reciprocal diffusions.
\newblock {\em Journal of Mathematical Physics}, 37(2):769--802, 1996.

\bibitem{Nel67}
{E.} Nelson.
\newblock {\em Dynamical theories of Brownian motion}, volume~2.
\newblock Princeton university press Princeton, 1967.

\bibitem{privault2015large}
N.~Privault, X.~Yang, and J.C. Zambrini.
\newblock Large deviations for bernstein bridges.
\newblock {\em Stochastic Processes and their Applications}, 2015.

\bibitem{pruitt1964eigenvalues}
W.~Pruitt.
\newblock Eigenvalues of non-negative matrices.
\newblock {\em The Annals of Mathematical Statistics}, pages 1797--1800, 1964.

\bibitem{RT02}
{S.} R{\oe}lly and {M.} Thieullen.
\newblock A characterization of reciprocal processes via an integration by
  parts formula on the path space.
\newblock {\em Probability Theory and Related Fields}, 123(1):97--120, 2002.

\bibitem{RT05}
{S.} R{\oe}lly and {M.} Thieullen.
\newblock Duality formula for the bridges of a brownian diffusion: Application
  to gradient drifts.
\newblock {\em Stochastic Processes and their Applications},
  115(10):1677--1700, 2005.

\bibitem{Ross11}
{N.} Ross.
\newblock Fundamentals of {S}tein’s method.
\newblock {\em Probab. Surv}, 8:210--293, 2011.

\bibitem{royer2007initiation}
Gilles Royer.
\newblock {\em An initiation to logarithmic Sobolev inequalities}.
\newblock Number~5. American Mathematical Soc., 2007.

\bibitem{Simon2011}
B.~Simon.
\newblock {\em Convexity: An analytic viewpoint}, volume 187.
\newblock Cambridge University Press, 2011.

\bibitem{Th93}
M.~Thieullen.
\newblock Second order stochastic differential equations and non-{G}aussian
  reciprocal diffusions.
\newblock {\em Probability Theory and Related Fields}, 97(1-2):231--257, 1993.

\bibitem{wittich2005explicit}
O.~Wittich.
\newblock An explicit local uniform large deviation bound for brownian bridges.
\newblock {\em Statistics \& probability letters}, 73(1):51--56, 2005.

\bibitem{wu2000new}
L.~Wu.
\newblock A new modified logarithmic {S}obolev inequality for poisson point
  processes and several applications.
\newblock {\em Probability Theory and Related Fields}, 118(3):427--438, 2000.

\bibitem{yang2014large}
X.~Yang.
\newblock Large deviations for markov bridges with jumps.
\newblock {\em Journal of Mathematical Analysis and Applications},
  416(1):1--12, 2014.

\end{thebibliography}
\Addresses
\end{document}